\documentclass[11pt]{elsarticle}
\usepackage{titlecaps}
\Addlcwords{the of into via for and of on in an to hp-finite with}

\usepackage{float,graphicx,amsmath,amsthm,amssymb,amsfonts}
\newcommand{\bx}[0]{\mathbf{x}}
\newcommand{\by}[0]{\mathbf{y}}
\usepackage{lipsum,booktabs}
\usepackage{appendix}

\usepackage{titlecaps}
\Addlcwords{the of into via for and of on in an to hp-finite with}
\usepackage[textwidth=125pt,textsize=small]{todonotes}

\usepackage{changepage}
\newcommand{\CH}[0]{\mathcal{H}}
\newcommand{\CV}[0]{\mathcal{V}}
\newcommand{\RR}[0]{\mathbb{R}}

\newtheorem{theorem}{Theorem}
\newtheorem{remark}{Remark}
 \usepackage{booktabs}
 \usepackage{multirow}
\newtheorem{lemma}{Lemma}

\usepackage{hyperref}
\usepackage[nameinlink]{cleveref}
\usepackage{autonum}

\begin{document}
\begin{frontmatter}

\title{Spectral Method for the Fractional Laplacian in 2D and 3D}

\author[label1]{Kailai Xu}
 \ead{kailaix@stanford.edu}
 \address[label1]{Institute for Computational and Mathematical Engineering, Stanford University, Stanford, CA, 94305}
 \tnotetext[t1]{The first author thanks the Stanford Graduate Fellowship in Science \& Engineering and the 2018 Schlumberger Innovation Fellowship for the financial support.}
 
 \author[label2]{Eric Darve}
 \ead{darve@stanford.edu}
 \address[label2]{Mechanical Engineering, Stanford University, Stanford, CA, 94305}

\begin{abstract}
    A spectral method is considered for approximating the fractional Laplacian and solving the fractional Poisson problem in 2D and 3D unit balls. The method is based on the explicit formulation of the eigenfunctions and eigenvalues of the fractional Laplacian in the unit balls under the weighted $L^2$ space. The resulting method enjoys spectral accuracy for all fractional index $\alpha\in (0,2)$ and is computationally efficient due to the orthogonality of the basis functions. We also proposed a numerical integration strategy for computing the coefficients. Numerical examples in 2D and 3D are shown to demonstrate the effectiveness of the proposed methods.
\end{abstract}
\end{frontmatter}

\section{Introduction}

Fractional PDEs have attracted considerable attention recently due to its applications in soft matter~\cite{chen2004soft}, elasticity~\cite{dipierro2015dislocation}, turbulence~\cite{bakunin2008turbulence}, anomalous diffusion~\cite{bologna2000anomalous}, finance~\cite{woyczynski2001levy}, image denoising~\cite{gatto2015numerical},  porous media flow~\cite{vazquez2012nonlinear}, etc. Among different fractional operators, the fractional Laplacian has been intensively studied in the recent literature. For example, \cite{Chen2004} uses the fractional Laplacian for linear and nonlinear lossy media. \cite{cont2005finite} uses the fractional Laplacian for option pricing in jump diffusion and exponential L\'evy models. Recently, \cite{epps2018turbulence} provides the first ever derivation of the fractional Laplacian operator as a means to represent the mean friction in the turbulence modeling. For more application examples, see \cite{ros2014integro}.

Although the fractional Laplacian is studied extensively for operators on $\mathbb{R}^d$, the community has not yet reached an agreement on the definition of the fractional Laplacian on a bounded domain. We mention that two popular definitions are the \textit{spectral fractional Laplacian} and the \textit{integral fractional Laplacian}~(also called \textit{Dirichlet fractional Laplacian}), which are listed in \cref{tab:defs}~\cite{lin2016isogeometric}.  The spectral fractional Laplacian is defined by considering the eigenvalues and eigenfunctions of the Laplacian operator $-\Delta$ in $\Omega$ with zero Dirichlet boundary data on $\partial \Omega$ while the integral fractional Laplacian is defined by first performing a zero extension of the original function to $\RR^d$ and then using the fractional Laplacian definition on $\RR^d$.

\begin{table}[htbp]
\centering
\begin{tabular}{@{}p{1.6cm}p{5.5cm}p{4.5cm}@{}}
\toprule
                            & Spectral & Dirichlet \\ \midrule \\[-8pt]
Definition                  &    $
\begin{cases}
u(\bx) = \sum_{k=1}^\infty a_k \varphi_k(\bx)    \\
-(-\Delta)_S^{\alpha/2} u = \sum_{k=1}^\infty a_k\lambda_k^\alpha \varphi_k(\bx)
\end{cases}
$      &      \parbox{4cm}{$ (-\Delta)^{\alpha/2} u(\bx)$\quad $=\qquad c_{\alpha,d} 
     \mathrm{P.V.}\int_{\mathbb{R}^d} \frac{u(\bx) - u(\by)}{|\bx-\by|^{d+\alpha}}d\by$} \\[18pt]
\parbox{2cm}{Fractional Poisson Equation} $\qquad$ &    $
    \begin{cases} 
(-\Delta)^{\alpha/2} u(\bx) = f  & \mbox{in } \Omega \\
u(\bx) = 0 & \mbox{on } \partial\Omega
\end{cases}
$      &        $\begin{cases} 
(-\Delta)^{\alpha/2} u(\bx) = f  & \mbox{in } \Omega    \\
u(\bx) = 0 & \mbox{in } \Omega^c
\end{cases}$     \\[20pt]
 Note                     &    $(\lambda_k, \varphi_k)$ are eigenpairs of $\Delta$     &           \\\bottomrule
\end{tabular}
\caption{A comparison of two popular definitions of the fractional Laplacian on the bounded domain.}
\label{tab:defs}
\end{table}


%

In this paper, we will focus on the integral fractional Laplacian, for the reasons that will be stated soon. The integral fractional Laplacian is defined by  
\begin{equation} \label{eq:fl}
    (-\Delta)^{\alpha/2} u(\bx) : = c_{\alpha,d} 
     \mathrm{P.V.}\int_{\mathbb{R}^d} \frac{u(\bx) - u(\by)}{|\bx-\by|^{d+\alpha}}d\by
\end{equation}
where P.V.\ stands for the Cauchy principal value, and
\begin{equation}\label{equ:coe f}
    c_{\alpha,d} =\frac{{{2^\alpha }\Gamma \left( {\frac{{d + \alpha }}{2}} \right)}}{{{\pi ^{d/2}}\left| {\Gamma \left( -{\frac{\alpha }{2}} \right)} \right|}}
\end{equation}
The corresponding fractional Poisson problem is
\begin{equation}
  \begin{aligned} \label{equ:model}
(-\Delta)^{\alpha/2} u(\bx) &= f \qquad && \bx \in \Omega    \\
u(\bx) &= 0 && \bx \in \Omega^c
\end{aligned}
\end{equation}

Compared to the spectral definition, the integral definition requires the boundary condition to be defined on $\Omega^c$ instead of $\partial \Omega$. This definition can be derived from a killed L\'evy process in the bounded domain $\Omega$~\cite{duo2017comparative}. 

Up till now, no definition is better justified than the other without more assumptions. However, if we consider the fractional Laplacian as a way to describe long-range interactions~(or long tail effects), any effort to restrict the definition on a bounded domain is unsuitable. Indeed, the long-range interaction characterization makes it necessary to consider the whole domain $\RR^d$ instead of a bounded one, and the limits of the bounded domain definitions as $\Omega$ tends to $\RR^d$ coincide. Thus it is  reasonable to view the definitions on a bounded domain mentioned above as an approximation to the fractional Poisson problem
\begin{equation}\label{equ:ff}
    (-\Delta)^{\alpha/2} u(\bx) = f(\bx) \qquad  \bx \in \RR^d  
\end{equation}
where the support of the source term is restricted to a certain region and the domain of interest is bounded. The fractional Poisson equation with either the spectral or the integral fractional Laplacian tries to approximate \cref{equ:ff} by a problem defined on bounded domains. For the latter case, the value of $u(\bx)$ far away from the domain of interest is simply truncated to zero.

Viewed in this fashion, whichever  model we choose to approximate \cref{equ:ff} is equivalently valid for our purpose, if numerically dealt with carefully. However, in this article, we focus on \cref{equ:model}, for which we can propose an efficient and accurate spectral method. We will focus on 2D and 3D which are practically useful. In addition, throughout the paper, we focus on $\Omega = B_1^d(0)$, $d=$ 2, 3, the unit balls in $\RR^d$, which are most useful in practice. It is possible to generalize to $d\geq 4$. We also use the notation $f(x) \lesssim$ ($\gtrsim$) $g(x)$ to denote $f(x) = \mathcal{O}(g(x))$ ($g(x) = \mathcal{O}(f(x))$).

There have already been extensive theoretical investigations. However, only until recently has there been extensive literature that deals with the numerical computation of the nonlocal operator. Compared to the normal Laplacian operator, it presents many numerical challenges which will be discussed in the next section. 

For clarification, we list the notation used in the paper in \cref{tab:notation}.

\begin{table}[htpb]
\centering
\begin{tabular}{@{}ll@{}}
\toprule
Notation &  Description \\ \midrule
$(-\Delta)^{\alpha/2}$ &  The fractional Laplacian with index $\alpha/2$\\
$(-\Delta)_S^{\alpha/2}$ &  The spectral fractional Laplacian with index $\alpha/2$\\
P.V.\ & Principle value integration\\
$c_{\alpha,d}$ & Coefficient for the fractional Laplacian\\
$\Omega = B_1^d(0)$ & Unit ball in $\RR^d$~(an open set)\\
$\bar B_1^d(0)$ & Closure of the unit ball in $\RR^d$\\
${}_a^{RL}D_x$, ${}_x^{RL}D_b$ &  Left/Right Riemann-Liouville derivatives\\
$P_{l,m,n}(\bx)$ & Eigenfunctions of the fractional Laplacian \cref{equ:plmn} \\
$d^{\alpha,d}_{n,l}$ & Eigenvalues of the fractional Laplacian  \cref{equ:dnl}\\
$w(\bx)$ & Weight function $w(\bx) = (1-|\bx|)^{\frac\alpha2}_+$\\
$\alpha _{n,l}^{\alpha ,d}$ & Expansion coefficients of $u(\bx)$ in terms of $\{p_{l,m,n}\}$\\
$M_{d,l}$ & $M_{d,l} = \frac{d+2l-2}{d+l-2}\begin{pmatrix}
        d+l-2\\
        l
    \end{pmatrix}$, \cref{equ:mdl} \\
$p_{l,m,n}$ & $p_{l,m,n}=w(\bx)P_{l,m,n}(\bx)$ \\
$Y_l^m$ & Spherical harmonics\\
$Y_{l,m}$ & Real spherical harmonics \cref{equ:real} \\
$V_{l,m}$ & Solid spherical harmonics\\
$P_n^{(\alpha ,\beta )}(z)$ & Jacobi polynomials \cref{equ:jacobip}\\
$\mathbb{P}_K$ & Set of all polynomials in $\RR^d$ with degrees no larger than $K$\\
$\mathcal{P}_K$ & Space of homogeneous polynomials of degree $K\geq 0$\\
$L$ & Solution operator of \cref{equ:model}\\
$f_{NL}$ & Projection of $f$ onto $\{P_{l,m,n}\}_{0\leq l\leq L, 0\leq n\leq N, 1\leq m \leq M_{d,l}}$ \cref{equ:fnl}\\
$u_{NL}$ & $u_{NL} = Lf_{NL}$\\
$L_{NL}$ & Approximation solution operator of \cref{equ:model} $u_{NL} = Lf$\\
 \bottomrule
\end{tabular}
\caption{Notations frequently used in the paper}
\label{tab:notation}
\end{table}

\section{Fractional Poisson Problem}

There are several challenges in solving \cref{equ:model} numerically:
\begin{itemize}
    \item The \textit{principle value} integration of the integrand. This can be resolved in several ways. For example, \cite{2018arXiv180203770M} subtracts the singularity by the approximation of $u(\bx)-u(\by) \approx \sigma(|\bx-\by|) \sum_{|\beta|=1}(\bx-\by)^\beta D^\beta u(\by)$ where $\sigma$ is a window function with compact support and $1-\sigma(|\bx-\by|) = \mathcal{O}(|\bx-\by|^4), \bx\rightarrow \by$. The symmetry of the principle value integration is used to eliminate the first and third order derivatives. This method takes advantage of the $C^3$ continuity of the basis functions. \cite{duo2018novel} took a different approach and split the kernel into 
     \begin{equation}
         \frac{1}{|\bx-\by|^{d+\alpha}} = \frac{1}{|\bx-\by|^{\gamma}}\frac{1}{|\bx-\by|^{d+\alpha-\gamma}}
     \end{equation}
     in this way, they decomposed the principle value integration into a singular part which can be evaluated analytically and a nonsingular part which can be evaluated with regular quadratures. 
    \item The \textit{non-locality} behavior of the operator. Numerically this usually leads to a dense matrix which requires much more storage and has high computational cost. Compression techniques are usually applied to reduce the cost while sacrificing some accuracy. For example, \cite{ainsworth2017towards} applies panel clustering method to the finite element method and reduce both the storage and computation cost to {$\mathcal{O}(n\log^{2d} n)$}. However, the cost reduction is down by a trade-off with accuracy and achieving optimal convergence leads back to a full matrix. When the problem is solved on a rectangle domain and regular grid, fast transformation methods such as the FFT can be applied, and the cost is reduced to $\mathcal{O}(n\log n)$~\cite{2018arXiv180203770M, duo2018novel}.
    \item The reduced convergence rate for $C^{0,\alpha}$ solutions. For solutions that are less smooth, e.g., $C^{0,\alpha}$, contemporary methods usually suffer from reduced convergence, and the rates become $\alpha$-dependent. For example, the finite difference method in \cite{duo2018novel} can approximate the fractional Laplacian with order $\mathcal{O}(h^{1-\frac\alpha2})$, which converges slowly for $\alpha\approx 2$. The finite element method proposed in \cite{acosta2017fractional} only convergence like $\mathcal{O}(h^\frac12)$ in the energy space. The convergence rate can be improved to $\mathcal{O}({h})$ for $1<\alpha<2$ at the expense of adopting a graded mesh. The reduced convergence rate is mainly due to the reduced regularity of the solutions: indeed, a standard result for the fractional Poisson problem  \eqref{equ:model} is stated in \cref{thm:v-is-Calpha}, which generally says the solution will only be $C^{0,\alpha}$. The theorem also informs us that when designing numerical methods for the fractional Laplacian, we cannot simply assume that the solution is smooth. Instead, the ability to obtain good convergence for the $C^{0,\alpha}$ should serve as the metric for the efficiency of the numerical methods. 
    
Several methods that can lead to good convergence results exist. For example, in \cite{lu2017spectral}, the authors observed in 1D and $\Omega=[a,b]$, the integral fractional Laplacian operator takes the following form
    \begin{equation}\label{equ:flequiv}
        (-\Delta)^{\alpha/2}u(x) = \frac{{}_a^{RL}D_x^\alpha u(x) + {}_x^{RL}D_b^\alpha u(x) }{2\cos(\alpha\pi/2)}
    \end{equation}
where ${}_a^{RL}D_x$, ${}_x^{RL}D_b$ are the left/right Riemann-Liouville derivatives. Then the authors applied the Legendre collocation method to discretize the fractional Laplacian operator \cref{equ:flequiv} and obtained spectral accuracy. In \cite{acosta2018regularity} the authors also propose a spectral method for 1D problem, but from a different point of view. They take advantage of a set of explicit eigenfunctions and eigenvalues and use them as basis functions. This approach leads to an efficient and accurate method.
    
In \cite{khader2011numerical}, the authors applied the Chebyshev collocation method to the Caputo derivatives. \cite{hanert2010comparison} also considered the use of a Chebyshev spectral element method for the numerical approximation of the Riemann-Liouville derivatives. \cite{zayernouri2014fractional} proposed a spectral collocation method for left-sided steady-state fractional advection equation. In consideration of \cref{equ:flequiv}, these methods can also be adapted to the fractional Laplacian. 
    
We also mention that efforts have been made to solve \cref{equ:ff} directly. For example, in \cite{bueno2014fourier}, instead of studying the fractional Poisson equation in a bounded domain, they applied the Fourier spectral method to $\RR^d$, which leads to an easy-to-code approach.  This method is a workaround for the reduced regularity problem.
\end{itemize}

In this paper, we propose a spectral method that enjoys spectral accuracy for all $\alpha\in (0,2)$, overcoming the reduced convergence rate issue. This method is limited to a ball domain $B_1^d(0)$; but for approximation \cref{equ:ff} it is more than enough: we can always pick a large enough ball~(and scale it to $B_1^d(0)$) so that the numerical solution is a good approximation.  Due to its high accuracy, only a few terms are needed to achieve the same accuracy compared to the finite difference or finite element method. Besides, as we have explicit formulas for the fractional Laplacian of the basis functions, the principal value integration is no longer necessary.  This method can also be viewed as a generalization of \cite{acosta2018regularity}.

The codes for the paper will be publicly available at

{\footnotesize \url{https://github.com/kailaix/fractional_laplacian_spectral_method}}

\section{Preliminaries}

\subsection{Regularity of the Solution}

For the nonlocal equation \cref{equ:model}, a remarkable regularity result is \cite{ros2014dirichlet}

\begin{theorem}\label{thm:v-is-Calpha}
Let $\Omega$ be a bounded $C^{1,1}$ domain, $f\in L^\infty(\Omega)$, $u$ be a solution of \cref{equ:model}, and $\delta(x)={\rm dist}(x,\partial\Omega)$. Then, $u/\delta^{\alpha/2}|_\Omega$ can be continuously extended to $\overline\Omega$. 
Moreover,  we have $u/\delta^{{\alpha/2}}\in C^{s}(\overline{\Omega})$ and
\[ \|u/\delta^{\alpha/2} \|_{C^{s}(\overline \Omega)}\le C \|f\|_{L^\infty(\Omega)}\]
for some $s>0$ satisfying $s<\min\{\alpha/2,1-\alpha/2\}$. 
The constants $s$ and $C$ depend only on $\Omega$ and $\alpha$.
\end{theorem}

\begin{figure}[htpb]
\centering
\includegraphics[width=0.6\textwidth]{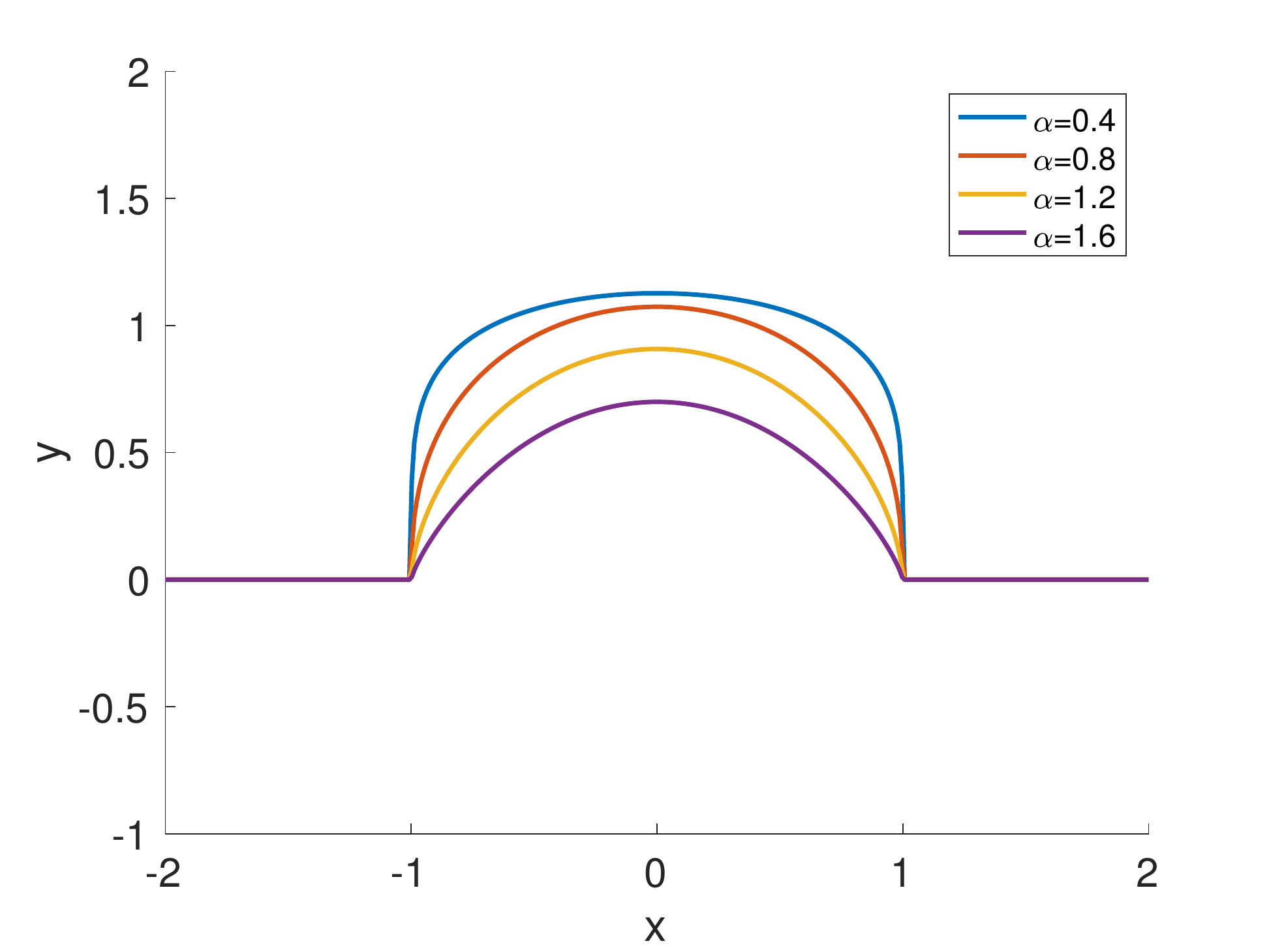}
\caption{Solutions to the fractional Poisson problem \cref{equ:model} with $f\equiv 1$, $x\in \Omega = [-1,1]$. We can clearly see the large slopes near $x=\pm 1$, indicating the discontinuity in the first order derivatives.}
\label{fig:rhsequalsone}
\end{figure}

This result indicates that $u$ behaves like $\delta^{\alpha/2}$ near the boundary. Besides, we can only expect H\"older continuity near the boundary and nothing better in general. In fact, this fractional behavior causes many of the current numerical methods to have reduced convergence rate for $C^{0,\alpha}$ or $C^{1,\alpha}$ solutions. \Cref{fig:rhsequalsone} shows the solutions to 1D fractional Poisson problem \cref{equ:model} with $f\equiv 1$, $x\in \Omega = [-1,1]$. We can clearly see the large slopes near $x=\pm 1$, indicating the discontinuity in the first order derivatives. These numerical methods try to fit the solution with polynomial basis functions. When the solution is smooth or has high order continuous derivatives, the approximation is appropriate. However, for the critical cases when the solution is only H\"older continuous, finer meshes or more terms are needed to achieve a preset accuracy. We illustrate the effect of reduced smoothness on the numerical solution in \cref{fig:fem}, where finite element methods are used to obtain the numerical solutions.

\begin{figure}[H]
\centering
\includegraphics[width=0.45\textwidth]{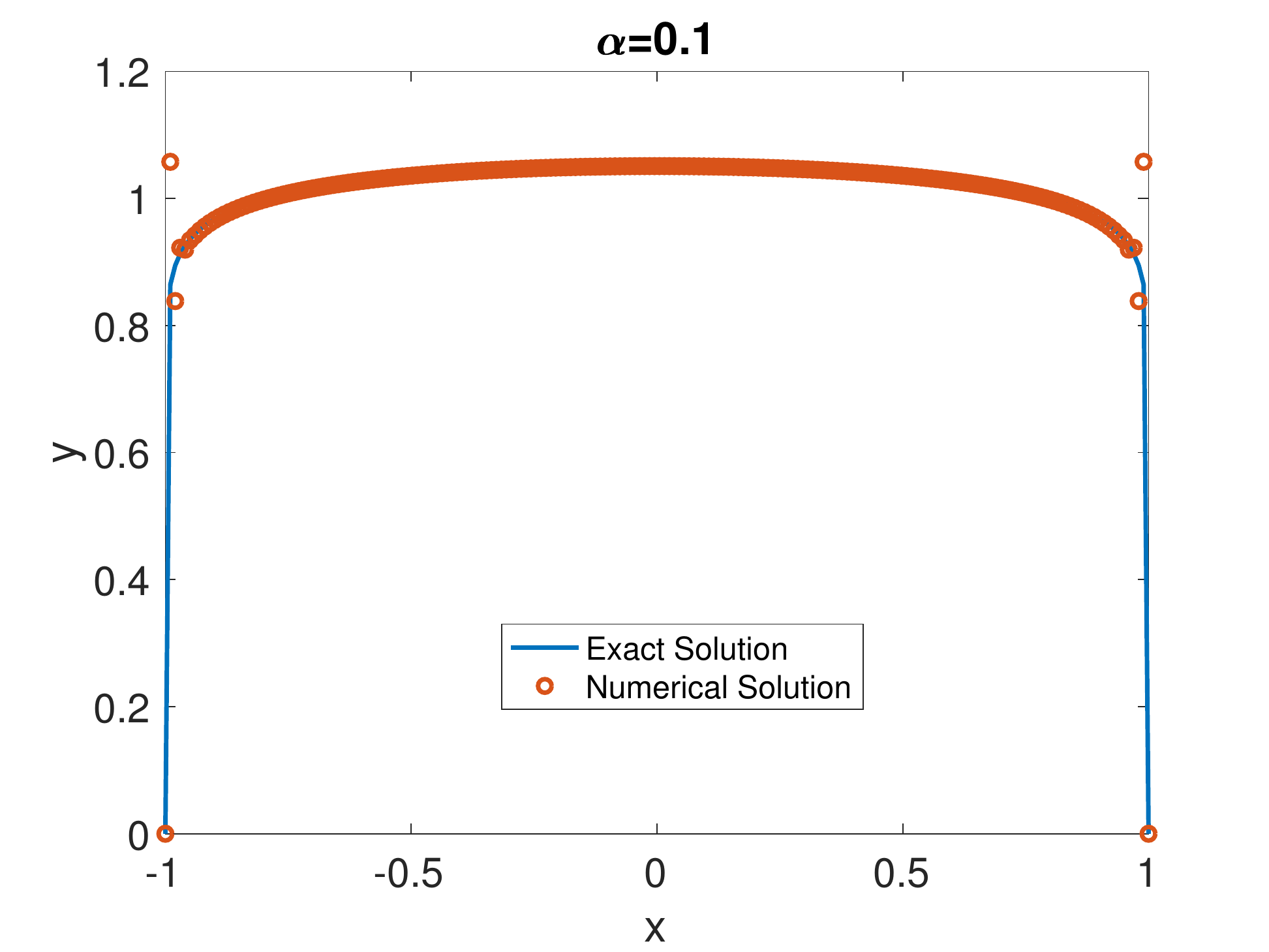}
\includegraphics[width=0.45\textwidth]{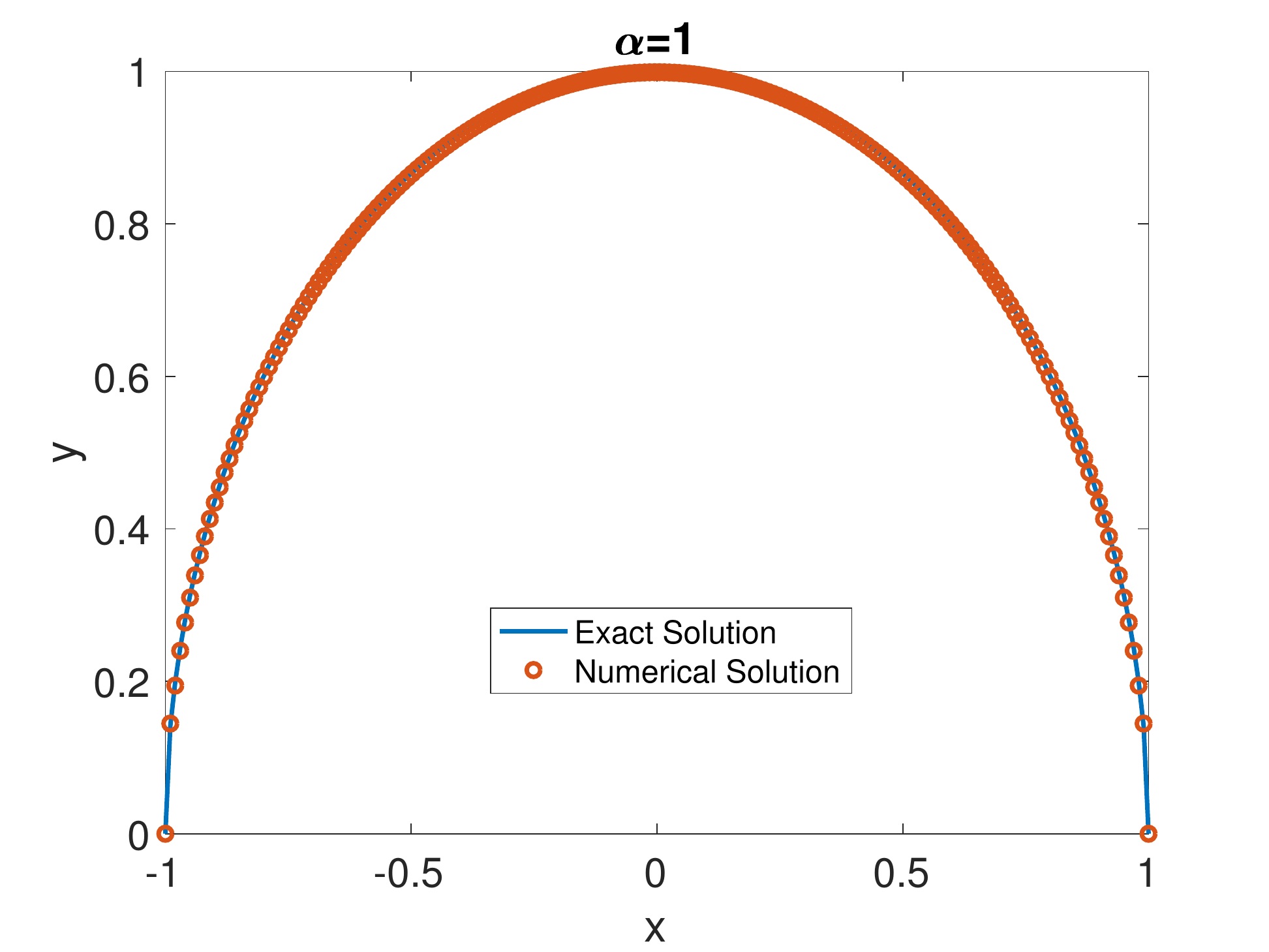}
\caption{Numerical results to \cref{equ:model} with $f\equiv 1$ using finite element methods. We use 200 uniform grid points on $[-1,1]$ and $\alpha=$0.1, 1 are used. We can clearly see that for small $\alpha$, where the solution is less smooth, there is severe deviation on the boundary for the numerical approximation.}
\label{fig:fem}
\end{figure}

As we have prior knowledge of how the solution behaves near the boundary, it is reasonable to build the asymptotic behavior into the representation of the solution. Fortunately, in the case of unit balls in 2D or 3D, we have explicit formulas for such representation, and therefore we can derive elegant and accurate spectral methods.

\subsection{Spectral Representation of the Solution}

Let $P_n^{(\alpha,\beta)}(z)$ be the Jacobi polynomial, e.g. 

\begin{equation}\label{equ:jacobip}
  P_n^{(\alpha ,\beta )}(z) = {{\Gamma (\alpha  + 1 + n)} \over {n!}}{}_2{\tilde F_1}\left( {\left. {\begin{matrix}
   { - n,1 + \alpha  + \beta  + n}  \cr 
   {\alpha  + 1}  \cr 

 \end{matrix} } \right|{{1 - z} \over 2}} \right)
\end{equation}
where ${}_p\tilde F_q$ is the regularized hypergeometric function~\cite{johansson2016computing}
\begin{equation}
    {}_p\tilde F_q\left( {\left. {\begin{array}{*{20}{c}}
{{a_1}, \ldots ,{a_p}}\\
{{b_1}, \ldots ,{b_q}}
\end{array}} \right|z} \right) = \sum_{k=0}^\infty \frac{(a_1)_k\ldots (a_p)_k    }{\Gamma(b_1+k)\ldots\Gamma(b_q+k)}\frac{z^k}{k!}, \quad (a)_k = \frac{\Gamma(a+k)}{\Gamma(a)}
\end{equation}

 Given $d\geq 1$ and $\alpha\in [0,2)$, we denote 
 \begin{equation}\label{equ:plmn}
    P_{l,m,n}(\bx) := V_{l,m}(\bx)P_n^{(\alpha /2,d/2 + l - 1)}(2|\bx{|^2} - 1)
\end{equation}
where $l,n\geq 0$ and $1\leq m \leq M_{d,l}$, and $V_{l,m}$ are \textit{solid harmonic polynomial}~\cite{dyda2017eigenvalues} in $\RR^d$ of degree $l\geq 0$, which is a homogeous polynomial of degree $l$ and harmonic, i.e., $\Delta V_{l,m}(\bx)=0$. We can also write $V_{l,m}=r^lY_l^m$~\cite{saborid2008coordinate}, where $Y_l^m$ is orthonormal with respect to the surface measure on the unit sphere, and is called \textit{spherical harmonics}~\cite{ferrers1877elementary}. Note that $P_{l,m,n}$ is a polynomial of variables $x_1,x_2,\ldots, x_d$. 
Here $M_{d,l}$ is given by
\begin{equation}\label{equ:mdl}
    M_{d,l} = \frac{d+2l-2}{d+l-2}\begin{pmatrix}
        d+l-2\\
        l
    \end{pmatrix}
\end{equation}
For example, we have
\begin{equation}
    M_{1,l} = 0, \, M_{2,l} = 2, \, M_{3,l} = 2l+1
\end{equation}

Finally we define 
\begin{equation}
    p_{l,m,n} = (1-|\bx|^2)^{\alpha/2}_+ P_{l,m,n}(\bx)
\end{equation}
We have the following theorem~\cite{dyda2017fractional}

\begin{theorem}\label{thm:main2}
    Assume that $\alpha>0, l, n\geq 0, 1\leq m\leq M_{d,l}$. Then 
    \begin{equation}
        (-\Delta)^{\alpha/2} p_{l,m,n}(\bx) = d^{\alpha,d}_{n,l}P_{l,m,n}(\bx),\quad \bx\in B_1^d(0)
    \end{equation}
    where
    \begin{equation}\label{equ:dnl}
        d^{\alpha,d}_{n,l} = {{{2^\alpha }\Gamma \left( {1 + {\alpha  \over 2} + n} \right)\Gamma \left( {{{\delta  + \alpha } \over 2} + n} \right)} \over {n!\Gamma \left( {{\delta  \over 2} + n} \right)}}, \quad \delta = d+2l
    \end{equation}
\end{theorem}

\begin{proof}
    See \cite{dyda2017fractional}. For an alternative proof when $d=1$, see \cite{acosta2018regularity}.
\end{proof}

This indicates that $P_{l,m,n}(\bx)$ is an eigenfunction with respect to the fractional Laplacian operator and weight function $w(\bx) = (1-|\bx|^2)_+^{\alpha/2}$, i.e., the operator $f\rightarrow (-\Delta)^{\frac\alpha2}(wf)$. The eigenvalues are exactly $d^{\alpha,d}_{n,l}$. To investigate the approximation property of $P_{l,m,n}(\bx)$, we need to introduce a sequence of the space of orthogonal polynomials of degree $n$, i.e., 
\begin{equation}
    \CV_n^d := \{P \in \Pi_n^d: \langle P, Q\rangle = 0, \quad\forall Q\in \Pi_{n-1}^d \}
\end{equation}
where $\Pi_n^d$ is the set of all polynomials in $\RR^d$ with degree at most $n$. Here $\langle\cdot, \cdot\rangle$ denotes a weighted inner product in $B_1^d(0)$; we denote the corresponding Hilbert space by $L^2(w)$, with weight function $w$, i.e.
\begin{equation}
    \int_{B_1^d(0)} p(\bx)q(\bx)w(\bx)d\bx = 0, \quad\forall p\in \CV_m^d, \ q\in\CV_n^d,\ m\neq n
\end{equation}

A notable result is that for every $n$, $P_{l,m,n}(\bx)$ forms a complete orthogonal system for $L^2(w)$,  $w=(1-|\bx|^2)^{\alpha/2}$. This is stated below~\cite{dunkl2014orthogonal}:
\begin{theorem}
    For $x , y \in \RR^d$, let $( x,  y ) : = x_1 y _1 + \cdots + x_d y_d$ and 
$|x|:= \sqrt{  (x , x )}$. On the unit ball $B^d_1(0) = \{x \in \RR^d: |x| \le 1\}$, 
consider 
\begin{equation}\label{ball-weight}
  W_\alpha(x,y) :=  \frac {\Gamma(\frac\alpha2 + \frac{d} {2} + 1)} {\pi^{d/2}
 \Gamma(\frac\alpha2+1)} (1-|\bx|^2)^{\frac\alpha2}, \quad \alpha>0,
\end{equation} 
normalized so that its integral over $B_1^d(0)$ is 1.

Let $\{Y_{m, l}: 1\le m 
\le M_{d,l}\}$, where $M_{d,l} = \dim \CH_{k}^d$ given by \cref{equ:mdl}, be an orthonormal 
basis of $\CH_{l}^d$, the space of spherical harmonics of degree $l$, with respect to 
the normalized surface measure. Define
\begin{equation}\label{ball-base2}
P_{l,m,n}(\bx) := [h_{l,n}]^{-1}
     P_{n}^{(\frac\alpha2, l + \frac{d-2}2)}(2|\bx|^2-1) V_{l,m}(\bx)
\end{equation}
where 
$$
[h_{l,n}]^2 = \frac{(\frac\alpha2+1)_{n} (\frac{d}{2})_{l+n} (l+n+\frac\alpha2+ \frac{d}{2})}
    { n ! (\frac\alpha2+\frac{d}{2}+1)_{l+n} (l+2n+\frac\alpha2+ \frac{d}{2})}
$$
Then for every $n$, $\{P_{l,m,n}:  0\leq l , 1 \le m \le M_{d,l}\}$ is an orthonormal basis 
of $L^2(w)$.
\end{theorem}
\begin{proof}
    See \cite{dunkl2014orthogonal} and its references.
\end{proof}
This enables us to develop a very accurate and efficient spectral method for the fractional Poisson problem.

\section{Spectral Method for the Fractional Poisson Problem }
In this section, we present the accurate and efficient spectral method for the fractional Poisson problem. We established the convergence result with proper assumptions on the right hand side.

Given \cref{thm:main2}, the spectral method for the fractional Poisson problem proceeds as follows:

\begin{itemize}
    \item From the expansion for the right hand side $f$
\begin{equation}
    f(\bx) = \sum\limits_{l = 0}^{ + \infty } {\sum\limits_{n = 0}^{ + \infty } {\sum\limits_{m = 1}^{{M_{d,l}}} {\alpha _{n,l}^{\alpha ,d}d_{n,l}^{\alpha ,d}{P_{l,m,n}}({\bf{x}})} } } 
\end{equation}
approximate $f$ with spectral accuracy: 
\begin{equation}
    f(\bx) \approx f_{NL}(\bx) = \sum\limits_{l = 0}^L {\sum\limits_{n = 0}^N {\sum\limits_{m = 1}^{{M_{d,l}}} {\alpha _{n,l}^{\alpha ,d}d_{n,l}^{\alpha ,d}{P_{l,m,n}}({\bf{x}})} } } 
\end{equation}
We control the approximation error with $N$, $L\geq 0$. Intuitively, $N$ controls the error along the radial direction while $L$ controls the error on the surface of spheres.

The approximation can be numerically done by projecting $f(\bx)$ onto the space spanned by $\{P_{l,m,n}\}$ using numerical quadratures
\begin{equation}\label{equ:alpha}
    \alpha _{n,l}^{\alpha ,d} = \frac{{\int_{{B_1}(0)} {f({\bf{x}}){P_{l,m,n}}({\bf{x}}){{(1 - |{\bf{x}}{|^2})}^{\frac{\alpha }{2}}}d{\bf{x}}} }}{{\int_{{B_1}(0)} {{P_{l,m,n}}{{({\bf{x}})}^2}{{(1 - |{\bf{x}}{|^2})}^{\frac{\alpha }{2}}}d{\bf{x}}} }}
\end{equation}

\item The solution is approximated as 
\begin{equation}
    {u_{NL}}({\bf{x}}) = \sum\limits_{l = 0}^L {\sum\limits_{n = 0}^N {\sum\limits_{m = 1}^{{M_{d,l}}} {\alpha _{n,l}^{\alpha ,d}(1 - |{\bf{x}}{|^2})_ + ^{\alpha /2}{p_{l,m,n}}({\bf{x}})} } } 
\end{equation}

\end{itemize}

The algorithm requires the computation of the integrals in the numerators and denominators in \cref{equ:alpha}. In \Cref{sect:quad}, we propose an efficient algorithm for this computation. The cost of the computation is $\mathcal{O}(NLM_{d,L})$. In \cref{sect:2d} and \cref{sect:3d} we give explicit formulas of \cref{equ:alpha} for 2D and 3D cases, which will be most useful for practical use. 

In the following discussion, we use $Lf(\bx)$ to denote the  exact solution of \cref{equ:model} and $L_{NL}f(\bx)$ to denote the numerical solution obtained by the procedure described above. We immediately obtain a regularity result for \cref{equ:alpha}:
\begin{lemma}\label{lemma:l0}
    Assume $f\in L^2(w)$, then the solution to \cref{equ:model} satisfies $Lf\in L^2(\Omega)$ and
    \begin{equation}
        \|Lf\|_{L^2(\Omega)}\lesssim \|f\|_{L^2(w)}
    \end{equation}
\end{lemma}

\begin{proof}
    As $f\in L^2(w)$, we can write $f$ as an expansion
    \begin{equation}
    f(\bx) = \sum\limits_{l = 0}^{ + \infty } {\sum\limits_{n = 0}^{ + \infty } {\sum\limits_{m = 1}^{{M_{d,l}}} {\alpha _{n,l}^{\alpha ,d}d_{n,l}^{\alpha ,d}{P_{l,m,n}}({\bf{x}})} } } 
\end{equation}
and therefore
\begin{equation}\label{equ:p1}
    \|f\|_{L^2(w)}^2=\sum\limits_{l = 0}^{ + \infty } {\sum\limits_{n = 0}^{ + \infty } {\sum\limits_{m = 1}^{{M_{d,l}}} {{{\left( {\alpha _{n,l}^{\alpha ,d}d_{n,l}^{\alpha ,d}} \right)}^2}\mathop 
    \int \nolimits_\Omega ^{} {P_{l,m,n}}{{({\bf{x}})}^2}{{(1 - |{\bf{x}}{|^2})}^{\alpha /2}}d{\bf{x}}} } } 
\end{equation}
The solution is given by 
\begin{equation}
Lf(\bx) = \sum\limits_{l = 0}^{ + \infty } {\sum\limits_{n = 0}^{ + \infty } {\sum\limits_{m = 1}^{{M_{d,l}}} {\alpha _{n,l}^{\alpha ,d}(1 - |{\bf{x}}{|^2})_ + ^{\alpha /2}{P_{l,m,n}}({\bf{x}})} } } 
\end{equation}
and
\begin{multline}\label{equ:p2}
    \|Lf(\bx)\|_{L^2(\Omega)}=\sum\limits_{l = 0}^{ + \infty } {\sum\limits_{n = 0}^{ + \infty } {\sum\limits_{m = 1}^{{M_{d,l}}} {{{\left( {\alpha _{n,l}^{\alpha ,d}} \right)}^2}\int_\Omega ^{} {(1 - |{\bf{x}}{|^2})_ + ^\alpha {P_{l,m,n}}{{({\bf{x}})}^2}d{\bf{x}}} } } } \\ \le \sum\limits_{l = 0}^{ + \infty } {\sum\limits_{n = 0}^{ + \infty } {\sum\limits_{m = 1}^{{M_{d,l}}} {{{\left( {\alpha _{n,l}^{\alpha ,d}} \right)}^2}\int_\Omega ^{} {(1 - |{\bf{x}}{|^2})_ + ^{\alpha /2}{P_{l,m,n}}{{({\bf{x}})}^2}d{\bf{x}}} } } } 
\end{multline}

Compare \cref{equ:p1,equ:p2} we conclude
\begin{equation}
C    \|Lf(\bx)\|_{L^2(\Omega)}\leq \|f\|_{L^2(w)}
\end{equation}
where $C$ can be picked as
\begin{equation}
    C = \min_{l,n} d^{\alpha,d}_{n,l}
\end{equation}

Since  $d^{\alpha,d}_{n,l}>0$ and
\begin{equation}
    {{\Gamma \left( {1 + {\alpha  \over 2} + n} \right)\Gamma \left( {{{\delta  + \alpha } \over 2} + n} \right)} \over {n!\Gamma \left( {{\delta  \over 2} + n} \right)}} > 1
\end{equation}
we deduce $C=\min\limits_{l,n} d^{\alpha,d}_{n,l}\geq 2^\alpha>0$.
\end{proof}

In what follows, we will give explicit formulas in 2D and 3D for the convenience of computation and analysis . 
\subsection{Explicit Formulas in 2D}\label{sect:2d}


If we have an expansion for $u(r,\theta)$, 
\begin{equation}\label{equ:uexpansion}
    u(r, \theta) = (1-r^2)^{\frac{\alpha}{2}} \sum_{n=0}^\infty \sum_{l=0}^\infty (c_{l0}^n \cos(l\theta) + c_{l1}^n \sin(l\theta)) P_n^{(\frac{\alpha}{2},l)}(2r^2-1) r^l
\end{equation}
then the fractional Laplacian can be computed directly from
\begin{equation}\label{equ:fexapnsion}
    f(r, \theta) = (-\Delta)^{\frac{\alpha}{2}} u(r, \theta) = \sum_{n=0}^\infty \sum_{l=0}^\infty d_{n,l}^{\alpha,2} (c_{l0}^n \cos(l\theta) + c_{l1}^n \sin(l\theta)) P_n^{(\frac{\alpha}{2},l)}(2r^2-1) r^l
\end{equation}

The coefficients in \cref{equ:uexpansion} can be computed by 
\begin{equation}\label{equ:utof}
    \left\{
  \begin{aligned}
  c_{l0}^n &= {{\int_0^{2\pi } d \theta \int_0^1 u (r,\theta )\cos (l\theta )P_n^{\left( {{\alpha  \over 2},l} \right)}(2{r^2} - 1){r^{l + 1}}dr} \over {(1+\delta_l) \pi \int_0^1 {P_n^{\left( {{\alpha  \over 2},l} \right)}} {{(2{r^2} - 1)}^2}{r^{2l + 1}}{{(1 - {r^2})}^{{\alpha  \over 2}}}dr}}, \quad l\geq 0\\
  c_{l1}^n &= {{\int_0^{2\pi } d \theta \int_0^1 u (r,\theta )\sin (l\theta )P_n^{\left( {{\alpha  \over 2},l} \right)}(2{r^2} - 1){r^{l + 1}}dr} \over {\pi \int_0^1 {P_n^{\left( {{\alpha  \over 2},l} \right)}} {{(2{r^2} - 1)}^2}{r^{2l + 1}}{{(1 - {r^2})}^{{\alpha  \over 2}}}dr}},l \ge 1\\
  c_{01}^n &= 0
\end{aligned}\right.
\end{equation}
where $\delta_l = \left\{\begin{matrix}
    0 & l\neq 0 \\
    1 & l = 0
\end{matrix} \right.$

Given $f(r,\theta)$, the coefficients in \cref{equ:uexpansion} can be computed by

\begin{equation}\label{equ:ftou}
\left\{
  \begin{aligned}
  c_{l0}^n &= {{\int_0^{2\pi } d \theta \int_0^1 f (r,\theta )\cos (l\theta )P_n^{\left( {{\alpha  \over 2},l} \right)}(2{r^2} - 1){r^{l + 1}}{{(1 - {r^2})}^{{\alpha  \over 2}}}dr} \over {(1+\delta_l)\pi d_{n,l}^{\alpha ,2}\int_0^1 {P_n^{\left( {{\alpha  \over 2},l} \right)}} {{(2{r^2} - 1)}^2}{r^{2l + 1}}{{(1 - {r^2})}^{{\alpha  \over 2}}}dr}}, \quad l\geq 0\\
  c_{l1}^n &= {{\int_0^{2\pi } d \theta \int_0^1 f (r,\theta )\sin (l\theta )P_n^{\left( {{\alpha  \over 2},l} \right)}(2{r^2} - 1){r^{l + 1}}{{(1 - {r^2})}^{{\alpha  \over 2}}}dr} \over {\pi d_{n,l}^{\alpha ,2}\int_0^1 {P_n^{\left( {{\alpha  \over 2},l} \right)}} {{(2{r^2} - 1)}^2}{r^{2l + 1}}{{(1 - {r^2})}^{{\alpha  \over 2}}}dr}},l \ge 1\\
  c_{01}^n &= 0
\end{aligned}\right.
\end{equation}
Then applying \cref{equ:uexpansion} we are able to compute $u(r,\theta)$.

When $u(r,\theta)$ or $f(r,\theta)$ is a radial function, $f(r,\theta)$ or $u(r,\theta)$ is also radial. In addition, we see that $c_{lk}^n = 0, \forall l\geq 1, n\geq 0,  k=0,1$, thus \cref{equ:fexapnsion,equ:ftou,equ:uexpansion,equ:utof} are greatly simplified:
\begin{align}
    u(r) &= {(1 - {r^2})^{{\alpha  \over 2}}}\mathop \sum \limits_{n = 0}^\infty  c_{00}^nP_n^{\left( {{\alpha  \over 2},0} \right)}(2{r^2} - 1)\label{equ:uc} \\
    f(r)& = \mathop \sum \limits_{n = 0}^\infty  d_{n,l}^{\alpha ,2}c_{00}^nP_n^{\left( {{\alpha  \over 2},0} \right)}(2{r^2} - 1)\label{equ:fc}
\end{align}
and the corresponding coefficients are
\begin{equation}\label{equ:c00}
    c_{00}^n = {{\int_0^1 {f(r)rP_n^{\left( {{\alpha  \over 2},l} \right)}(2{r^2} - 1)} {{(1 - {r^2})}^{{\alpha  \over 2}}}dr} \over {d_{n,0}^{\alpha ,2}\int_0^1 {rP_n^{\left( {{\alpha  \over 2},l} \right)}{{(2{r^2} - 1)}^2}{{(1 - {r^2})}^{{\alpha  \over 2}}}dr} }} = {{\int_0^1 {u(r)rP_n^{\left( {{\alpha  \over 2},l} \right)}(2{r^2} - 1)} dr} \over {\int_0^1 {rP_n^{\left( {{\alpha  \over 2},l} \right)}{{(2{r^2} - 1)}^2}{{(1 - {r^2})}^{{\alpha  \over 2}}}dr} }}
\end{equation}

\subsection{Explicit Formulas in 3D}\label{sect:3d}


In this case, we have $M_{3,l}=2l+1$. For convenience, we let the index $m$ range be $-l,-l+1,\ldots,l-1,l$. Let $Y_m^l$ be the spherical harmonics
\begin{equation}
    Y^m_l(\theta,\phi) ={( - 1)^m}\sqrt {{{2l + 1} \over {4\pi }}{{(l - m)!} \over {(l + m)!}}} P_l^m(\cos \theta ){e^{im\phi }}
\end{equation}
be the spherical harmonics. For example, in the case when $l=0$, we have 
\begin{equation}
    Y_0^0(\theta, \phi) = \frac{1}{\sqrt{4\pi}}
\end{equation}

We assume the solution has the following expansion
\begin{equation}\label{equ:uexpansion3}
    u(r,\theta, \phi) = (1-r^2)^{\frac{\alpha}{2}}\sum\limits_{n = 0}^\infty  {\sum\limits_{l = 0}^\infty  {\sum\limits_{m =  - l}^l {c_{l,m}^n{r^l}Y_l^m(\theta ,\phi )P_n^{\left( {{\alpha  \over 2},{1 \over 2} + l} \right)}(2{r^2} - 1)} } } 
\end{equation}
then
\begin{equation}\label{equ:fexpansion3}
\begin{split}
    f(r,\theta,\phi) & = \Re (-\Delta)^{\frac{\alpha}{2}} u(r,\theta,\phi) \\
    & = \sum\limits_{n = 0}^\infty  {\sum\limits_{l = 0}^\infty  {\sum\limits_{m =  - l}^l {d_{n,l}^{n,\alpha }c_{l,m}^n{r^l}Y_l^m(\theta ,\phi )P_n^{\left( {{\alpha  \over 2},{1 \over 2} + l} \right)}(2{r^2} - 1)} } } 
\end{split}
\end{equation}
where $\Re$ denotes the real part.

Then coefficients in \cref{equ:uexpansion3} can be expressed through $u(r,\theta,\phi) $ by
\begin{equation}
    c_{l,m}^n = {{\int_0^{2\pi } {d\theta \int_0^\pi  {d\phi \int_0^1 {u(r,\theta ,\phi ){r^{l + 2}}Y_l^m(\theta ,\phi )^* P_n^{\left( {{\alpha  \over 2},{1 \over 2} + l} \right)}(2{r^2} - 1)\sin \phi dr} } } } \over {\int_0^{2\pi } {d\theta \int_0^\pi  {d\phi \int_0^1 {{r^{2l + 2}}{{(1 - {r^2})}^{{\alpha  \over 2}}}|Y_l^m{{(\theta ,\phi )}|^2}P_n^{\left( {{\alpha  \over 2},{1 \over 2} + l} \right)}{{(2{r^2} - 1)}^2}\sin \phi dr} } } }}
\end{equation}
or through $f(r,\theta,\phi) $ by
\begin{equation}
    c_{l,m}^n = {{\int_0^{2\pi } {d\theta \int_0^\pi  {d\phi \int_0^1 {f(r,\theta ,\phi ){r^{l + 2}}{{(1 - {r^2})}^{{\alpha  \over 2}}}Y_l^m(\theta ,\phi )^* P_n^{\left( {{\alpha  \over 2},{1 \over 2} + l} \right)}(2{r^2} - 1)\sin \phi dr} } } } \over {d_{l,m}^{\alpha ,3}\left( {\int_0^{2\pi } {\int_0^\pi  |{Y_l^m{{(\theta ,\phi )}|^2}\sin \phi d\phi d\theta } } } \right)\left( {\int_0^1 {{r^{2l + 2}}{{(1 - {r^2})}^{{\alpha  \over 2}}}P_n^{\left( {{\alpha  \over 2},{1 \over 2} + l} \right)}{{(2{r^2} - 1)}^2}dr} } \right)}}
\end{equation}

Note that $c_{l,m}^n$ is a complex number, and we need to take the conjugate of $Y_l^m(\theta ,\phi )$ in the numerator. 

When $u(r,\theta,\phi)$ or $f(r,\theta,\phi)$ is a radial function, $f(r,\theta,\phi)$ or $u(r,\theta,\phi)$ is also radial. And we have the following simplifications:
\begin{align}
    u(r) &= \mathop \sum \limits_{n = 0}^\infty  c_{00}^n{1 \over {\sqrt {4\pi } }}P_n^{\left( {{\alpha  \over 2},{1 \over 2} + l} \right)}(2{r^2} - 1){(1 - {r^2})^{{\alpha  \over 2}}}\\
    f(r) &= \mathop \sum \limits_{n = 0}^\infty  d_{n0}^{\alpha ,3}c_{00}^n{1 \over {\sqrt {4\pi } }}P_n^{\left( {{\alpha  \over 2},{1 \over 2} + l} \right)}(2{r^2} - 1)\\
    c_{00}^n &= {{\sqrt {4\pi } } \over {d_{n,0}^{\alpha ,3}}}{{\int_0^1 {f(r){r^2}P_n^{\left( {{\alpha  \over 2},{1 \over 2}} \right)}(2{r^2} - 1){{(1 - {r^2})}^{{\alpha  \over 2}}}dr} } \over {\int_0^1 {P_n^{\left( {{\alpha  \over 2},{1 \over 2}} \right)}{{(2{r^2} - 1)}^2}{r^2}{{(1 - {r^2})}^{{\alpha  \over 2}}}dr} }}\\
    &=\sqrt {4\pi } {{\int_0^1 {u(r){r^2}P_n^{\left( {{\alpha  \over 2},{1 \over 2}} \right)}(2{r^2} - 1)dr} } \over {\int_0^1 {P_n^{\left( {{\alpha  \over 2},{1 \over 2}} \right)}{{(2{r^2} - 1)}^2}{r^2}{{(1 - {r^2})}^{{\alpha  \over 2}}}dr} }}
\end{align}

\subsection{Error Analysis}

To carry out the error analysis for the proposed spectral method, we first consider the approximation properties of the basis functions $\{P_{l,m,n}(\bx)\}$. A typical approach is to compare it with the polynomials basis.

Denote $\mathbb{P}_K$ be the set of all polynomials in $\RR^d$ with degrees no larger than $K$, and $\mathcal{P}_K$ the space of homogeneous polynomials of degree $K\geq 0$, i.e., the vector space spanned by 
\begin{equation}
	x_1^{\alpha_1} x_2^{\alpha_2} \ldots x_d^{\alpha_d}, \quad \sum_{i=1}^d \alpha_i = K
\end{equation}

We define the subspace $H_K\subset \mathcal{P}_K$ the space spanned by $V_{K,m}(\RR^d)$, $m=1,2,\ldots, M_{d,l}$. A remarkable result is given in \cite{spherical:online}
\begin{lemma}\label{lemma:har}
Assume $n\geq 0$, we have the decomposition
	$$\mathcal{P}_d = H_d\oplus r^2 H_{d-2} \oplus r^4 H_{d-4}\oplus \ldots$$
\end{lemma}

\begin{proof}
	See \cite{spherical:online}.
\end{proof}

To illustrate, consider $d=3$, then we have 
\begin{equation}
	\mathcal{P}_3 = \mathrm{span}\{x^3, y^3, z^3, x^2y, x^2z, xy, y^2z, xz^2, yz^2, xyz\}
\end{equation}
By \cref{lemma:har}, we have $\mathcal{P}_3 = H_3 \oplus r^2 H_1$, and $H_1$, $H_3$ can be read from \cref{tab:spherical}. In the table, the real spherical harmonics are defined by 
\begin{equation}\label{equ:real}
  Y_{lm} =
\begin{cases}
\frac{1}{\sqrt{2}} ( Y_l^m + (-1)^mY_l^{-m} ) & \text{if } m > 0 \\
Y_l^m        & \text{if } m = 0 \\
\frac{1}{i \sqrt{2}}( Y_l^{-m} - (-1)^mY_l^m)       & \text{if } m < 0
\end{cases}
\end{equation}

\begin{table}[htbp]
\centering

\begin{tabular}{@{}llll@{}}
\toprule
                       & $m$ & $Y_l^m$                                                      & $V_{l,m}$                            \\ \midrule
$l=0$                  & 0   & $\sqrt{\frac{1}{4\pi}}$                                      & $\sqrt{\frac{1}{4\pi}}$            \\ \midrule
\multirow{3}{*}{$l=1$} & $-1$   & $\sqrt{\frac{3}{4\pi}}\sin\psi\sin\theta$                    & $\sqrt{\frac{3}{4\pi}}x$           \\
                       & 0   & $\sqrt{\frac{3}{4\pi}}\cos\theta$                            & $\sqrt{\frac{3}{4\pi}}z$           \\
                       & 1   & $\sqrt{\frac{3}{4\pi}}\cos\psi\sin\theta$                    & $\sqrt{\frac{3}{4\pi}}y$           \\ \midrule
\multirow{5}{*}{$l=2$} & $-2$   & $\sqrt{\frac{15}{4\pi}}\sin\psi\cos\psi\sin^2\theta$         & $\sqrt{\frac{15}{4\pi}}xy$         \\
                       & $-1$  & $\sqrt{\frac{15}{4\pi}}\sin\psi\cos\psi\cos\theta$           & $\sqrt{\frac{15}{4\pi}}yz$         \\
                       & 0  & $\sqrt{\frac{5}{16\pi}}(3\cos^2\theta-1)$                    & $\sqrt{\frac{5}{16\pi}}(2z^2-x^2-y^2)$   \\
                       & 1   & $\sqrt{\frac{15}{4\pi}}\cos\psi\sin\theta\cos\theta$         & $\sqrt{\frac{15}{8\pi}}xz$         \\
                       & 2   & $\sqrt{\frac{15}{16\pi}}(\cos^2\psi-\sin^2\psi)\sin^2\theta$ & $\sqrt{\frac{15}{32\pi}}(x^2-y^2)$ \\ \bottomrule
\end{tabular}
\caption{Real spherical harmonics $Y_{lm}$ and solid spherical harmonics $V_{l,m}$ in 3D for $l=0,1,2$.}
\label{tab:spherical}
\end{table}

We now clarify the relationship between the polynomial basis and the basis $\{P_{l,m,n}\}$ in the unit ball.

\begin{lemma}\label{lemma:l1}
    Let 
\begin{equation}
    \mathcal{S}_K:=\mathrm{span}\{P_{l,m,n}\}_{n\leq K, \, l\leq 2K, \, 0\leq m\leq M_{d,l}}
\end{equation}
then $\mathbb{P}_K \subset \mathcal{S}_K$. Recall $P_{l,m,n}$ is defined as \cref{equ:plmn}
\begin{equation}
    P_{l,m,n}(\bx) := V_{l,m}(\bx)P_n^{(\alpha /2,d/2 + l - 1)}(2|\bx{|^2} - 1)
\end{equation}
\end{lemma}

\begin{proof}
We prove the result by mathematical induction.

When $K=0$, $\mathbb{P}_0=\mathcal{S}_0=$ constant polynomials. When $K=1$, $\mathbb{P}_1=\mathrm{span}\{1,x_1,x_2,\ldots,x_d\}$. Note $x_i$ are all valid solutions to $\Delta u(\bx)=0, \bx\in B_1^d(0)$ thus $\mathrm{span}\{V_{1,m}\}_{1\leq m\leq M_{1,d}}=\mathrm{span}\{x_1,x_2,\ldots,x_d\}$ and therefore 
\[ \mathrm{span}\{V_{0,0}, V_{1,m}\}_{1\leq m\leq M_{1,d}}=\mathrm{span}\{1,x_1,x_2,\ldots,x_d\}=\mathbb{P}_1
\]
implying $\mathbb{P}_1 \subset \mathcal{S}_1$.

 Assume the claim is true for all integers smaller than $K\geq 1$. Let $x_1^{\alpha_1}x_2^{\alpha_2}\ldots x_d^{\alpha_d}$ be a monomial with $\sum_{i=1}^d\alpha_i=K+1, \alpha_i\in\mathbb{N}$, we prove $x_1^{\alpha_1}x_2^{\alpha_2}\ldots x_d^{\alpha_d}\in \mathcal{S}_{K+1}$. 

According to \cref{lemma:har}, there exists a polynomial $f(\bx)\in \mathbb{P}_{K-1}$, s.t.
\begin{equation}
    x_1^{\alpha_1}x_2^{\alpha_2}\ldots x_d^{\alpha_d} = \underbrace{\sum_{m=1}^{M_{d,K+1}} C_m V_{K+1,m}}_{\in H_{K+1}} + |\bx|^2 f(\bx)
\end{equation}
where $C_m\in\RR$.

As $|\bx|^2f(\bx) = \frac{(2r^2-1)+1}{2}f(\bx) \in \mathcal{S}^{K+1}$, where $r=|\bx|$, we deduce that $x_1^{\alpha_1}x_2^{\alpha_2}\ldots x_d^{\alpha_d}\in \mathcal{S}^{K+1}$. Therefore the lemma is true.
\end{proof}

Then we state an approximation result for the polynomials in the unit ball. The approximation result is essentially stating that we can approximate a $C^k$ function with a polynomial of degree $n$ and obtain the error rate $\mathcal{O}(n^{-k})$.

\begin{lemma}[Theorem 3.4, \cite{ragozin1971constructive}]\label{lemma:l2}
    Given $f\in C^k(\bar B^d_1(0))$, define 
    \begin{align}
        w(f;h) &= \sup\{|f(x)-f(y)|:|x-y|\leq h\}\\
        w(f^{(k)};h) &= \sum_{|\beta| = k} w(\partial^\beta f; h)
    \end{align}
    then there exist polynomials $p_n$ with degree at most $n$, such that
    \begin{equation}
        \|f-p_n\|_{L^\infty(B^d_1(0))}\leq D(k,d)n^{-k}(n^{-1}\|f\|_{C^k(B^d_1(0))} + w(f^{(k)};1/n))
    \end{equation}
    where $D(k,d)$ is a constant which depends only on $k$ and $m$.
\end{lemma}
\begin{proof}
    See \cite{ragozin1971constructive}.
\end{proof}

Finally, we can give an error estimate of the proposed method.

\begin{theorem}\label{thm:main}
Assume $f\in  C^r(\bar\Omega)$, $N\geq K, L\geq 2K$, where $K\in\mathbb{N}$, then there exists a constant $C>0$, only depending on $K,d$ such that the error
\begin{equation}
    \|u_{NL}-u\|_{L^2(\Omega)} \leq Cn^{-r}\|f\|_{C^r(\Omega)}
\end{equation}
where $u_{NL} = Lf_{NL}$ and $f_{NL}$ is the projection of $f$ onto $\{P_{l,m,n}\}_{0\leq l\leq L, 0\leq n\leq N, 1\leq m \leq M_{d,l}}$, i.e., 
\begin{align}
	f(\bx) &= \sum\limits_{l = 0}^{ + \infty } {\sum\limits_{n = 0}^{ + \infty } {\sum\limits_{m = 1}^{{M_{d,l}}} {\alpha _{n,l}^{\alpha ,d}d_{n,l}^{\alpha ,d}{P_{l,m,n}}({\bf{x}})} } }\\
	f_{NL}(\bx) &= \sum\limits_{l = 0}^{ L } {\sum\limits_{n = 0}^{ N } {\sum\limits_{m = 1}^{{M_{d,l}}} {\alpha _{n,l}^{\alpha ,d}d_{n,l}^{\alpha ,d}{P_{l,m,n}}({\bf{x}})} } } \label{equ:fnl}
\end{align}
\end{theorem}
\begin{proof}
Let $q_K$ be the particular polynomial in \cref{lemma:l2} for $n=K$. By \cref{lemma:l0,lemma:l2} we have
    \begin{align}
        \|u_{NL}-u\|_{L^2(\Omega)} &= \|Lf_{NL}-Lf\|_{L^2(\Omega)} \\
        &\lesssim \|f_{NL}-f\|_{L^2(w)}\\
        &\leq \|q_K-f\|_{L^2(w)}\\
        &\leq \|q_K-f\|_{L^\infty(\Omega)}\label{equ:reuse} \\
        & \lesssim n^{-r}\|f\|_{C^{r}(\bar\Omega)}
    \end{align}
    
    The second inequality is due to \cref{lemma:l1} and the fact that $f_{NL}$ is the projection of $f$ onto $\{P_{l,m,n}\}_{0\leq l\leq L, 0\leq n\leq N, 1\leq m \leq M_{d,l}}$. 
\end{proof}

We now prove the spectral accuracy of the proposed algorithm. Let $\chi=[-1,1]^d$ be a hyper-square in $\RR^d$ which encloses $\Omega$.

We define the generalized Bernstein ellipse by
\begin{align}\label{eq-genB}
B(\mathcal{X},\rho):=B([-1,1],\rho_1)\times\cdots\times B([-1,1],\rho_D ),
\end{align}
where $B([-1,1],\rho_1)$ is the Bernstein ellipse.

\begin{theorem}[Theorem 2.1, \cite{glau2016improved}] 
 \label{Asymptotic_error_decay_multidim_combined}
  Let $f:\mathcal{X}\rightarrow\mathbb{R}$ have an analytic extension to some generalized Bernstein ellipse $B(\mathcal{X},\rho)$ for some parameter vector $\rho\in (1,\infty)^d$ with $\max_{x\in B(\mathcal{X},\rho)}|f(\bx)|\le V<\infty$. 
  
  Let $N = (N_1, N_2,\ldots, N_d)$ and $\bar N = \sum_{i=1}^d N_i$. Then there exists a polynomial $p_{\bar N}(\bx)$ with degree at most $\bar N$
  \begin{align}
   \max_{x\in\mathcal{X}}\big|f(\bx)& - p_{\bar N}(\bx)\big|\le\min\{a(\rho,N,d),b(\rho,N,d)\},
  \end{align}
  where, denoting by $S_d$ the set of all permutations on $d$ elements,
  \begin{align}
  a(\rho,N,d) &=
  \min_{\sigma\in S_d}\sum_{i=1}^d 4V\frac{\rho_{\sigma(i)}^{-N_i}}{\rho_i-1} + \sum_{k=2}^d 4V\frac{\rho_{\sigma(k)}^{-N_k}}{\rho_{\sigma(k)}-1} \, 2^{k-1}  \frac{(k-1) + 2^{k-1}-1}{\prod_{j=1}^{k-1}(1-\frac{1}{\rho_{\sigma(j)}})}\\
  b(\rho,N,d)&=2^{\frac{d}{2}+1} \; V \; \Big(\sum_{i=1}^d\rho_i^{-2N_i}\prod_{j=1}^d\frac{1}{1-\rho_j^{-2}}\Big)^{\frac{1}{2}}
  \end{align}
  \end{theorem}
\begin{remark}
    The polynomial can be chosen as 
    \begin{equation}
        p_{\bar N} (\bx) = \sum_{j\in J} c_jT_j(\bx)
    \end{equation}
    The summation index $j$ is a multi-index over $J=\{(j_1,j_2,\ldots,j_d):j_i\in \mathbb{N}, j_i\leq \lfloor K/d\rfloor \}$. For $j=(j_1,j_2,\ldots,j_D)$, $T_j$ is defined by 
    \begin{equation}
        T_j(x_1,\ldots,x_d) = \Pi_{i=1}^d T_{j_i}(x_i)
    \end{equation}
    where $T_i$ are Chebyshev polynomials. The coefficients are 
    \begin{equation}
        {c_j} = \left( {\Pi _{i = 1}^d\frac{{{2^{1[0 < {j_i} < {N_i}]}}}}{{{N_i}}}} \right){\sum _{{k_1} = 0}^{{N_1}}}{\vphantom{\sum}}'' \ldots {\sum_{{k_d} = 0}^{{N_d}}}{\vphantom{\sum}}'' f({x^{({k_1},{k_2}, \ldots ,{k_d})}})\Pi _{i = 1}^d\cos \left( {{j_i}\pi \frac{{{k_i}}}{{{N_i}}}} \right)
    \end{equation}
    where $\sum''$ indicates that the first and last summand are halved. The Chebyshev nodes are given by $x_{k_i}=\cos\left(\pi\frac{k_i}{N_i} \right)$.
\end{remark}

\begin{theorem}
    Let $f\in B(\mathcal{X},\underline\rho)$, where $\underline \rho = (\rho, \rho,\ldots, \rho)$, and $N\geq K, L\geq 2K, K\in\mathbb{N}$ then we have the error estimate 
    \begin{equation}
        \|u_{NL}-u\|_{L^2(\Omega)} \lesssim \frac{\rho^{-\lfloor \frac{K}{d}\rfloor}}{(\rho-1)(1-\frac1\rho)^{d-1}} \|f\|_{L^\infty(B(\mathcal{X},\rho))}
    \end{equation}
\end{theorem}

\begin{proof}
    The proof can be adapted from \cref{thm:main} by picking $q_K=p_{\bar N}(\bx)$. Let $V=\max_{x\in B(\chi,\rho)}$ and $N_i =\lfloor \frac{K}{d}\rfloor$, then according to \cref{Asymptotic_error_decay_multidim_combined}, we have
    \begin{equation}
        \max_{x\in \chi} |f(\bx)-p_{\bar N}(\bx)| \lesssim \frac{\rho^{-\lfloor \frac{K}{d}\rfloor}V}{\rho-1}  + \frac{\rho^{-\lfloor \frac{K}{d}\rfloor}V}{(\rho-1)(1-\frac1\rho)^{d-1}}\lesssim \frac{\rho^{-\lfloor \frac{K}{d}\rfloor}V}{(\rho-1)(1-\frac1\rho)^{d-1}}
    \end{equation}

    Since $\sum_{i=1}^d N_i = \sum_{i=1}^d \lfloor \frac{K}{d}\rfloor \leq K $, $p_{\bar N}(\bx)$ is a polynomial with degree at most $K$.  Therefore, a similar argument to \cref{thm:main} yields
    \begin{align}
        \|u_{NL}-u\|_{L^2(\Omega)} &\lesssim \|f_{NL}-f\|_{L^2(w)}\lesssim \|p_{\bar N}(\bx)-f\|_{L^2(w)}\\
        &\lesssim \|u(\bx)-p_{\bar N}(\bx)\|_{L^\infty(\Omega)} \\
        &\lesssim \frac{\rho^{-\lfloor \frac{K}{d}\rfloor}}{(\rho-1)(1-\frac1\rho)^{d-1}} \|f\|_{L^\infty(B(\mathcal{X},\rho))}
    \end{align}
\end{proof}

\begin{remark}
The result is consistent with \cite{acosta2018regularity} for the 1D case. However, for 2D and 3D, we must be careful how \cref{equ:alpha} is computed to avoid unnecessary error. This issue will be tackled immediately in the next section. 
\end{remark}

\subsection{Quadrature for weight function $w(\bx) = (1-|\bx|^2)^{\frac{\alpha}{2}}$}\label{sect:quad}

In the equation above, we need to compute the integral of the following form
\begin{equation}\label{equ:integral}
    I = \int_0^1 f(r) (1-r^2)^{\frac{\alpha}{2}} dr
\end{equation}
where $f(r)$ is a polynomial.

Note the standard Gauss-Jacobi quadrature cannot be applied to \cref{equ:integral} directly as the integration limit is $[0,1]$ instead of $[-1,1]$. Here we derive a numerical quadrature based on Lanczos process~\cite{warn720026:online}. This is a variant of the Golub-Welsch algorithm~\cite{golub1969calculation}.

We first approximate the integral \cref{equ:integral} with some simple quadrature rules such as the middle point rule:
\begin{equation}\label{equ:form}
    I \approx I_w = \sum_{i=1}^N f(x_i) w_i
\end{equation}
where $N$ is a very large number such that $|I-I_w|$ is negligible, and 
\begin{equation}
    w_i = \frac{(1-r_i^2)^{\frac{\alpha}{2}} }{N}
\end{equation} 
is the weight. Let
\begin{equation}
    A = \begin{bmatrix}
        x_1 & & & \\
           & x_2 & & \\
           & & &\ddots & \\
           & & & & x_N
    \end{bmatrix}\!, \;
    b = \begin{bmatrix}
        \sqrt{w_1}\\
        \sqrt{w_2}\\
        \vdots\\
        \sqrt{w_N}
    \end{bmatrix}
\end{equation}
thus we can write \cref{equ:form} as
\begin{equation}
    I_w = b^T f(A) b
\end{equation}



Now we first run a Lanczos process on $A$ with starting vector $b$ for $K$ iterations, and obtain $A\approx Q_K T_K Q_K^T$, where $T_K\in\mathbb{R}^{K\times K}$, $K\ll N$, is a symmetric tridiagonal matrix and $Q_K^TQ^K=I$. Then we apply the eigenvalue decomposition to $T_K$
\begin{equation}
    T_K = V\Lambda V^T
\end{equation}
Thus $A\approx (Q_K V) \Lambda (Q_K V)^T$. 

We now have
\begin{equation}\label{equ:quad}
    I_s = b^TC f(\Lambda) C^Tb \approx I_w
\end{equation}

Therefore, we can extract the new quadrature rules from \cref{equ:quad} by identifying $\{\lambda_i\}_{i=1}^K=\textrm{diag}(\Lambda)$ as the nodes and $\{s_i\}_{i=1}^K=(C^Tb)^2$ as weights. To justify that the quadrature is indeed effective, we prove the following theorem.

\begin{theorem}
    Assume $|I-I_w|\leq \varepsilon$ and the new quadrature nodes and weights $\{\lambda_i\}_{i=1}^K, \{s_i\}_{i=1}^K$ are obtained as described above. Then for all polynomials $f$ with degree $2K-1$ or less, we have
    \begin{equation}
        I_s = \sum_{i=1}^N f(\lambda_i) s_i = I_w
    \end{equation}
    and therefore
    \begin{equation}
        |I-I_s| \leq \varepsilon
    \end{equation}
\end{theorem}

Therefore, we reduce the total number of evaluation points for $f(r)$ from $N$ to $K\ll N$ without sacrificing much accuracy~(precisely the same for polynomials with degree $2K-1$ or less).

\begin{proof}
    Let $\mathcal{K}_K(A,b)=\{b, Ab, \ldots, A^{K-1}b\}$ be the Krylov subspace. Then by the property of the Lanczos process,  $\forall x\in \mathcal{K}_K(A,b)$
    \begin{equation}
        Ax = Q_KT_KQ_K^T x = C\Lambda C^Tx
    \end{equation}
    
    If $f$ is a polynomial of degree $2K-1$ or less, we have
    \begin{equation}
        b^Tf(A) b = b^TCf(\Lambda ) C^Tb 
    \end{equation}
    This indicates the new quadrature gives the same value as the original quadrature if the integrand is $f$, i.e., $I_w=I_s$, thus
    \begin{equation}
        |I_s-I| = |I_w-I|\leq \varepsilon
    \end{equation}
\end{proof}

\section{Numerical Experiments}

In this section, we carry out the numerical experiments of the spectral method in 2D and 3D. Both radial and non-radial functions are considered. All these experiments show spectral accuracy for the method.  We will also apply our method to solve fractional diffusion equations. We use the notation $B$ to denote the unit ball in the corresponding dimension.

\subsection{Accuracy of the Fractional Laplacian Estimate}\label{sect:1}

We consider the function
\begin{equation}\label{equ:testcase1}
    \begin{aligned}
        u(\bx) &= (1-|\bx|^2)^{s+\frac{\alpha}{2}} & \bx \in B\ \\
        u(\bx) &= 0 & \bx \in B^c
    \end{aligned}
\end{equation}
where $s=0,1,2,\ldots$, which is in the $C^{s, \frac{\alpha}{2}}$ class. As 
\begin{equation}
    \frac{u(\bx)}{(1-|\bx|^2)^{\frac{\alpha}{2}}} =(1-|\bx|^2)^s
\end{equation}
is a polynomial in $\bx$, it is expected that the orthogonal expansion can represent the solution exactly with finite number of terms. The result is shown in \cref{tab:2d}. We can see that the error demonstrates spectral accuracy for all the cases. The error is measured by computing the fractional Laplacian  using a high degree expansion~(in our case, we use $n=5$) as approximation to $(-\Delta)^{\frac{\alpha}{2}} u$ and then the infimum error is computed by 
\begin{equation}
    \|(-\Delta)^{\frac{\alpha}{2}} u - (-\Delta)^{\frac{\alpha}{2}}_n u\|_\infty
\end{equation}

\begin{table}[htbp]
\centering
\begin{adjustwidth}{-1cm}{}
\begin{tabular}{l|ccccc}
\toprule
  & $n=0$ & $n=1$ & $n=2$ & $n=3$ & $n=4$\\ 
  \midrule
  $s=0$ & 4.5895{\footnotesize$\times 10^{-11}$} & 4.4276{\footnotesize $\times 10^{-11}$} & 4.0085{\footnotesize $\times 10^{-11}$} & 3.2115{\footnotesize $\times 10^{-11}$} & 1.9156{\footnotesize $\times 10^{-11}$}\\
$s=1$ & 1.0086 & 5.0678{\footnotesize $\times 10^{-12}$} & 4.5871{\footnotesize $\times 10^{-12}$} & 3.6737{\footnotesize $\times 10^{-12}$} & 2.1904{\footnotesize $\times 10^{-12}$}\\
$s=2$ & 1.187 & 0.46785 & 3.6082{\footnotesize $\times 10^{-16}$} & 5.5511{\footnotesize $\times 10^{-17}$} & 1.6653{\footnotesize $\times 10^{-16}$}\\
$s=3$ & 1.4283 & 0.63502 & 0.17675 & 2.2204{\footnotesize $\times 10^{-16}$} & 2.2204{\footnotesize $\times 10^{-16}$}\\
\midrule
\end{tabular}

\medskip

\begin{tabular}{l|ccccc}
  & $n=0$ & $n=1$ & $n=2$ & $n=3$ & $n=4$\\ 
  \midrule
  $s=0$ & 3.7748{\footnotesize $\times 10^{-15}$} & 3.7748{\footnotesize $\times 10^{-15}$} & 4.2188{\footnotesize $\times 10^{-15}$} & 3.9968{\footnotesize $\times 10^{-15}$} & 5.5511{\footnotesize $\times 10^{-15}$}\\
$s=1$ & 2.1206 & 2.2204{\footnotesize $\times 10^{-15}$} & 1.7764{\footnotesize $\times 10^{-15}$} & 1.9984{\footnotesize $\times 10^{-15}$} & 4.4409{\footnotesize $\times 10^{-16}$}\\
$s=2$ & 2.272 & 1.3148 & 2.1649{\footnotesize $\times 10^{-15}$} & 2.6645{\footnotesize $\times 10^{-15}$} & 2.5535{\footnotesize $\times 10^{-15}$}\\
$s=3$ & 2.9125 & 1.5107 & 0.61323 & 3.2196{\footnotesize $\times 10^{-15}$} & 2.6923{\footnotesize $\times 10^{-15}$}\\
\midrule
\end{tabular}

\medskip

\begin{tabular}{l|ccccc}
  & $n=0$ & $n=1$ & $n=2$ & $n=3$ & $n=4$\\ 
  \midrule
  $s=0$ & 1.5289{\footnotesize $\times 10^{-10}$} & 1.5196{\footnotesize $\times 10^{-10}$} & 1.4687{\footnotesize $\times 10^{-10}$} & 1.3012{\footnotesize $\times 10^{-10}$} & 8.8244{\footnotesize $\times 10^{-11}$}\\
$s=1$ & 4.656 & 3.1748{\footnotesize $\times 10^{-12}$} & 3.066{\footnotesize $\times 10^{-12}$} & 2.7209{\footnotesize $\times 10^{-12}$} & 1.843{\footnotesize $\times 10^{-12}$}\\
$s=2$ & 4.6339 & 3.7373 & 1.1657{\footnotesize $\times 10^{-14}$} & 1.0991{\footnotesize $\times 10^{-14}$} & 9.77{\footnotesize $\times 10^{-15}$}\\
$s=3$ & 6.3058 & 4.1131 & 2.1158 & 1.2046{\footnotesize $\times 10^{-14}$} & 1.2879{\footnotesize $\times 10^{-14}$}\\
\bottomrule
\end{tabular}
\end{adjustwidth}
\caption{Numerical fractional Laplacian of \cref{equ:testcase1} for 2D. The corresponding $\alpha$ are 0.5, 1.0, 1.5 respectively. We can see that $n+1=s+1$ terms are need to compute the fractional Laplacian exactly.}
\label{tab:2d}

\end{table}

\begin{table}[htbp]
\centering

\begin{adjustwidth}{-1cm}{}
    
\begin{tabular}{l|ccccc}
\toprule
  & $n=0$ & $n=1$ & $n=2$ & $n=3$ & $n=4$\\ 
  \midrule
  $s=0$ & 5.1471{\footnotesize $\times 10^{-11}$} & 4.934{\footnotesize $\times 10^{-11}$} & 4.4336{\footnotesize $\times 10^{-11}$} & 3.5251{\footnotesize $\times 10^{-11}$} & 3.3248{\footnotesize $\times 10^{-11}$}\\
$s=1$ & 1.0574 & 5.6452{\footnotesize $\times 10^{-12}$} & 5.0717{\footnotesize $\times 10^{-12}$} & 4.0309{\footnotesize $\times 10^{-12}$} & 3.8007{\footnotesize $\times 10^{-12}$}\\
$s=2$ & 1.5068 & 0.50506 & 1.3323{\footnotesize $\times 10^{-15}$} & 8.8818{\footnotesize $\times 10^{-16}$} & 8.8818{\footnotesize $\times 10^{-16}$}\\
$s=3$ & 1.7771 & 0.92778 & 0.19824 & 6.6613{\footnotesize $\times 10^{-16}$} & 4.4409{\footnotesize $\times 10^{-16}$}\\
\midrule
\end{tabular}

\medskip

\begin{tabular}{l|ccccc}
  & $n=0$ & $n=1$ & $n=2$ & $n=3$ & $n=4$\\ 
  \midrule
  $s=0$ & 6.4393{\footnotesize $\times 10^{-15}$} & 7.1054{\footnotesize $\times 10^{-15}$} & 6.2172{\footnotesize $\times 10^{-15}$} & 7.5495{\footnotesize $\times 10^{-15}$} & 7.1054{\footnotesize $\times 10^{-15}$}\\
$s=1$ & 2 & 6.2172{\footnotesize $\times 10^{-15}$} & 5.7732{\footnotesize $\times 10^{-15}$} & 4.885{\footnotesize $\times 10^{-15}$} & 2.3315{\footnotesize $\times 10^{-15}$}\\
$s=2$ & 3.125 & 1.125 & 1.8319{\footnotesize $\times 10^{-15}$} & 2.4147{\footnotesize $\times 10^{-15}$} & 1.3323{\footnotesize $\times 10^{-15}$}\\
$s=3$ & 3.9375 & 2.1875 & 0.5 & 1.8041{\footnotesize $\times 10^{-15}$} & 4.4409{\footnotesize $\times 10^{-16}$}\\\midrule
\end{tabular}

\medskip

\begin{tabular}{l|ccccc}
  & $n=0$ & $n=1$ & $n=2$ & $n=3$ & $n=4$\\ 
  \midrule
  $s=0$ & 1.8478{\footnotesize $\times 10^{-10}$} & 1.8328{\footnotesize $\times 10^{-10}$} & 1.7625{\footnotesize $\times 10^{-10}$} & 1.5488{\footnotesize $\times 10^{-10}$} & 1.0392{\footnotesize $\times 10^{-10}$}\\
$s=1$ & 4.6974 & 3.8276{\footnotesize $\times 10^{-12}$} & 3.6797{\footnotesize $\times 10^{-12}$} & 3.2374{\footnotesize $\times 10^{-12}$} & 2.16{\footnotesize $\times 10^{-12}$}\\
$s=2$ & 6.8389 & 3.3634 & 6.3283{\footnotesize $\times 10^{-15}$} & 5.0515{\footnotesize $\times 10^{-15}$} & 1.1102{\footnotesize $\times 10^{-16}$}\\
$s=3$ & 9.1689 & 5.3726 & 1.7956 & 1.7764{\footnotesize $\times 10^{-15}$} & 1.138{\footnotesize $\times 10^{-15}$}\\
\bottomrule
\end{tabular}
\end{adjustwidth}

\caption{Numerical fractional Laplacian of \cref{equ:testcase1} for 3D. The corresponding $\alpha$ are 0.5, 1.0, 1.5 respectively. We can see that $n+1=s+1$ terms are need to compute the fractional Laplacian exactly.}
\end{table}

In addition, we consider the function
\begin{equation}\label{equ:testcase2}
    \begin{aligned}
        u(\bx) &= (1-|\bx|^2)^{s} & \bx \in B\ \\
        u(\bx) &= 0 &\bx \in B^c
    \end{aligned}
\end{equation}
where $s=1,2,\ldots$, which is in the $C^{s, 0}$ class. 
\[
\frac{u(\bx)}{(1-|\bx|^2)^{\frac{\alpha}{2}}} =(1-|\bx|^2)^{s-\frac{\alpha}{2}}
\]
cannot be exactly represented by a finite sum of the orthogonal system. It is interesting to look at the coefficients $c^n_{00}$ corresponding to $u(\bx)$. \Cref{fig:coefs} shows the case $s=1$, $\alpha=0.5$, and we see an algebraic decay of the coefficients. The decay is not exponential but algebraic due to the fractional nature of $\frac{u(\bx)}{(1-|\bx|^2)^{\frac{\alpha}{2}}}$. 

\begin{figure}[htbp]
\centering
\includegraphics[width=0.6\textwidth]{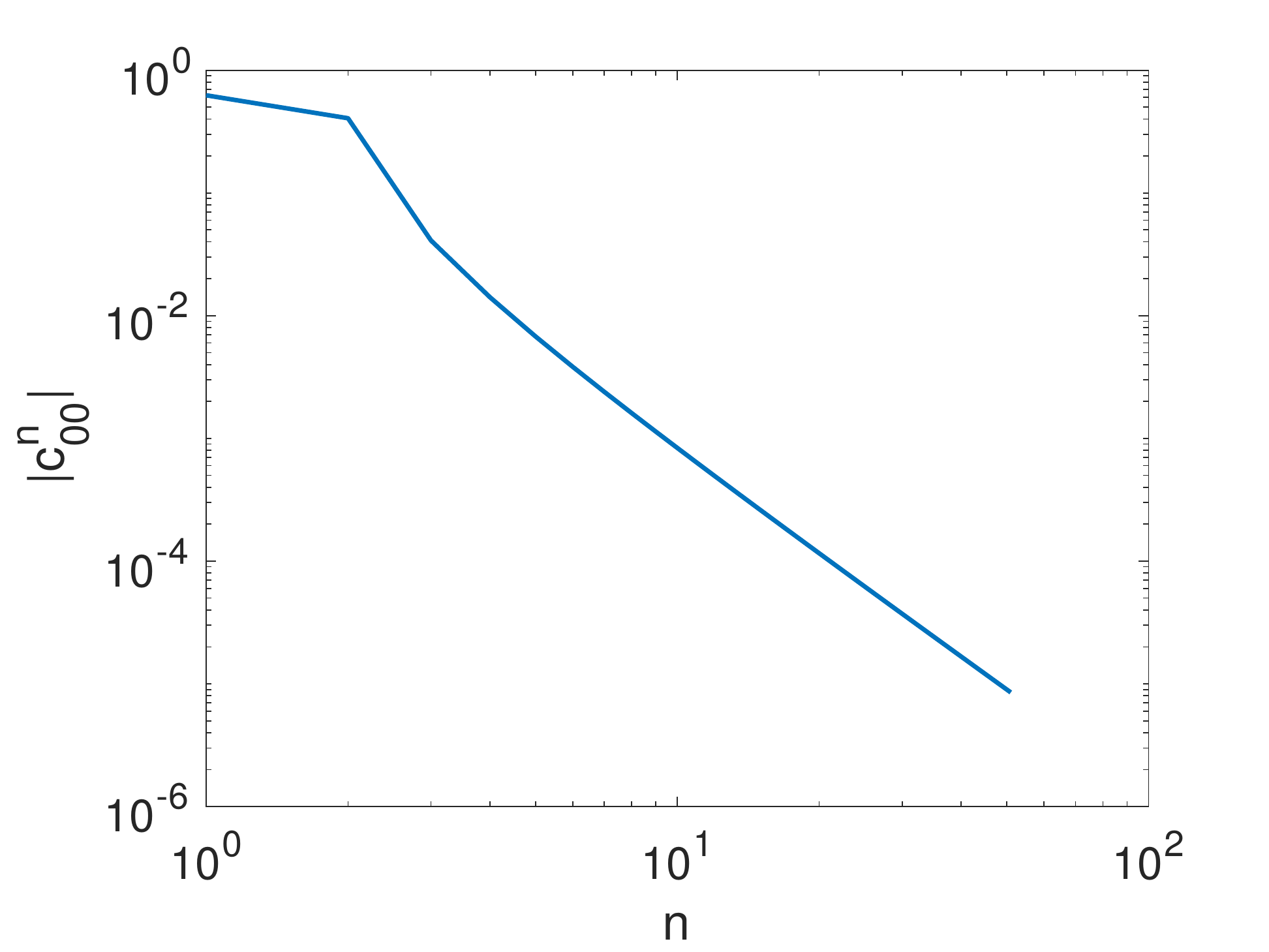}
\caption{Algebraic decay of coefficients $c_{00}^n$ for $u(\bx) = 1-|\bx|^2$ in 2D.}
\label{fig:coefs}
\end{figure}

We can still expect spectral accuracy as is shown in \cref{fig:ex}. In this experiment, we choose the numerical result of $n=50$ as the reference solution. We can see in all cases the solution convergences exponentially with respect to $n$. In the case $s=0$, the convergence is slow as the fractional power $(1-|\bx|^2)^{1-\frac{\alpha}{2}}$ where $1-\frac{\alpha}{2}<1$ is difficult to approximate with integer power polynomials.

\begin{figure}[htbp]
\centering
\includegraphics[width=0.45\textwidth]{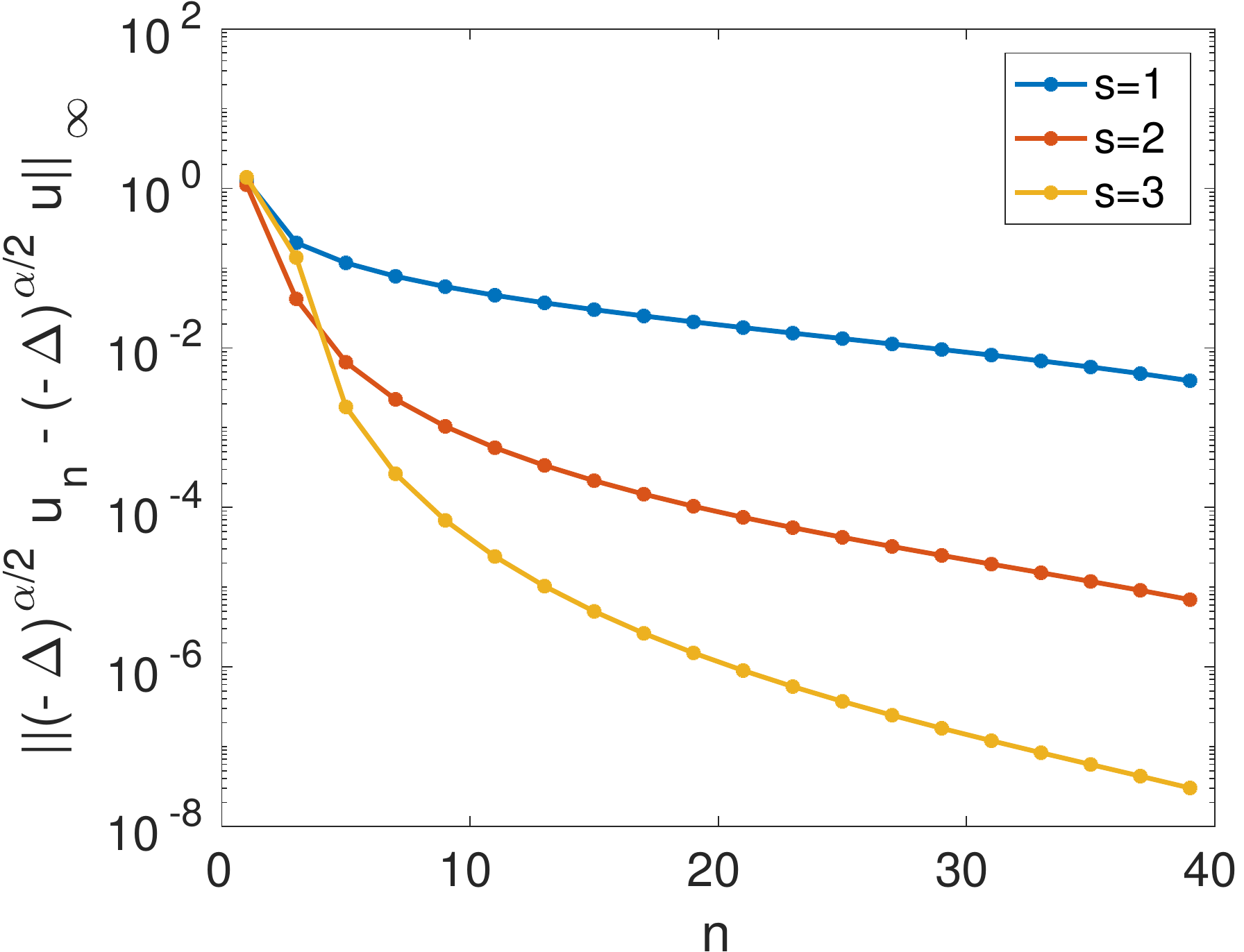}
\includegraphics[width=0.45\textwidth]{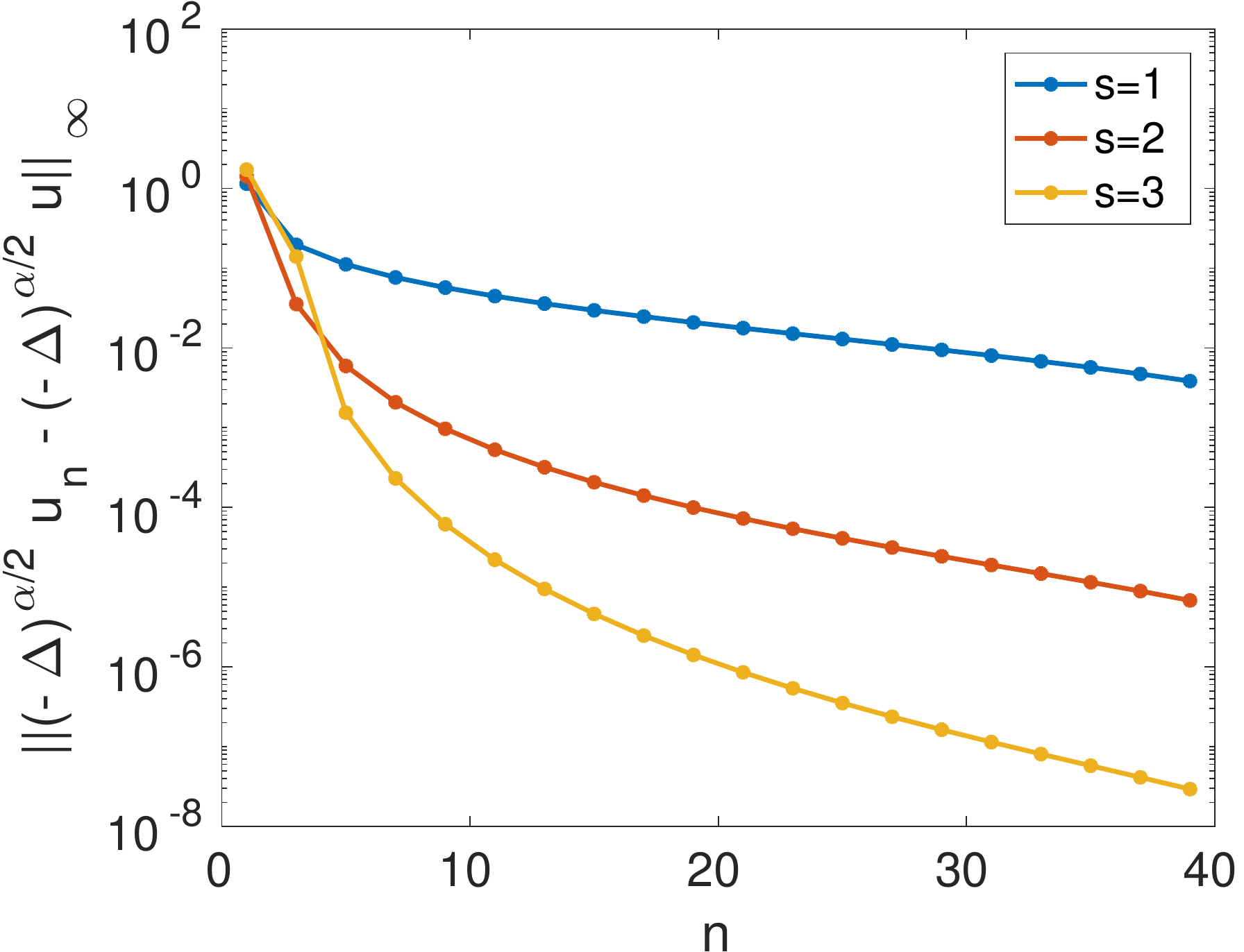}

\includegraphics[width=0.45\textwidth]{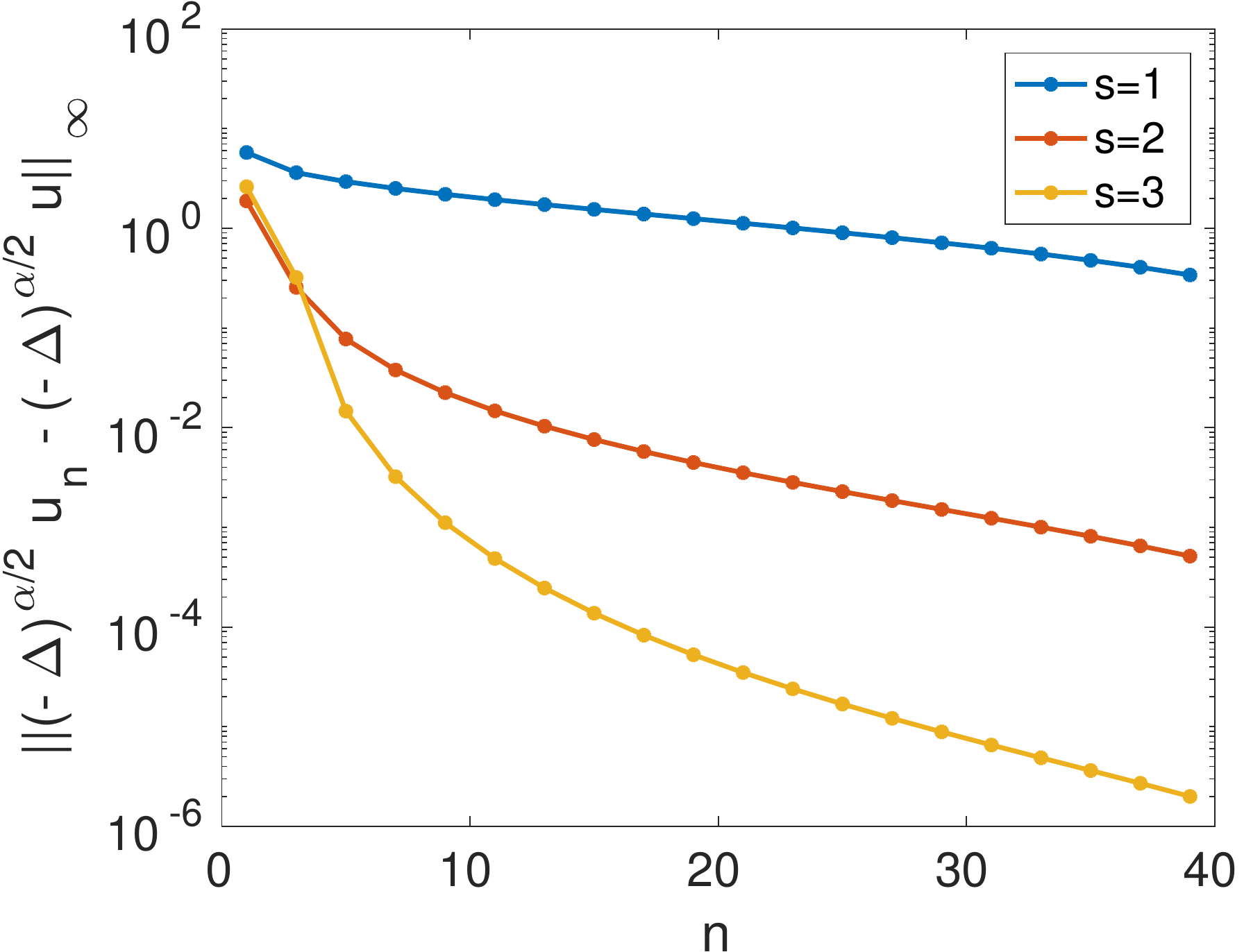}
\includegraphics[width=0.45\textwidth]{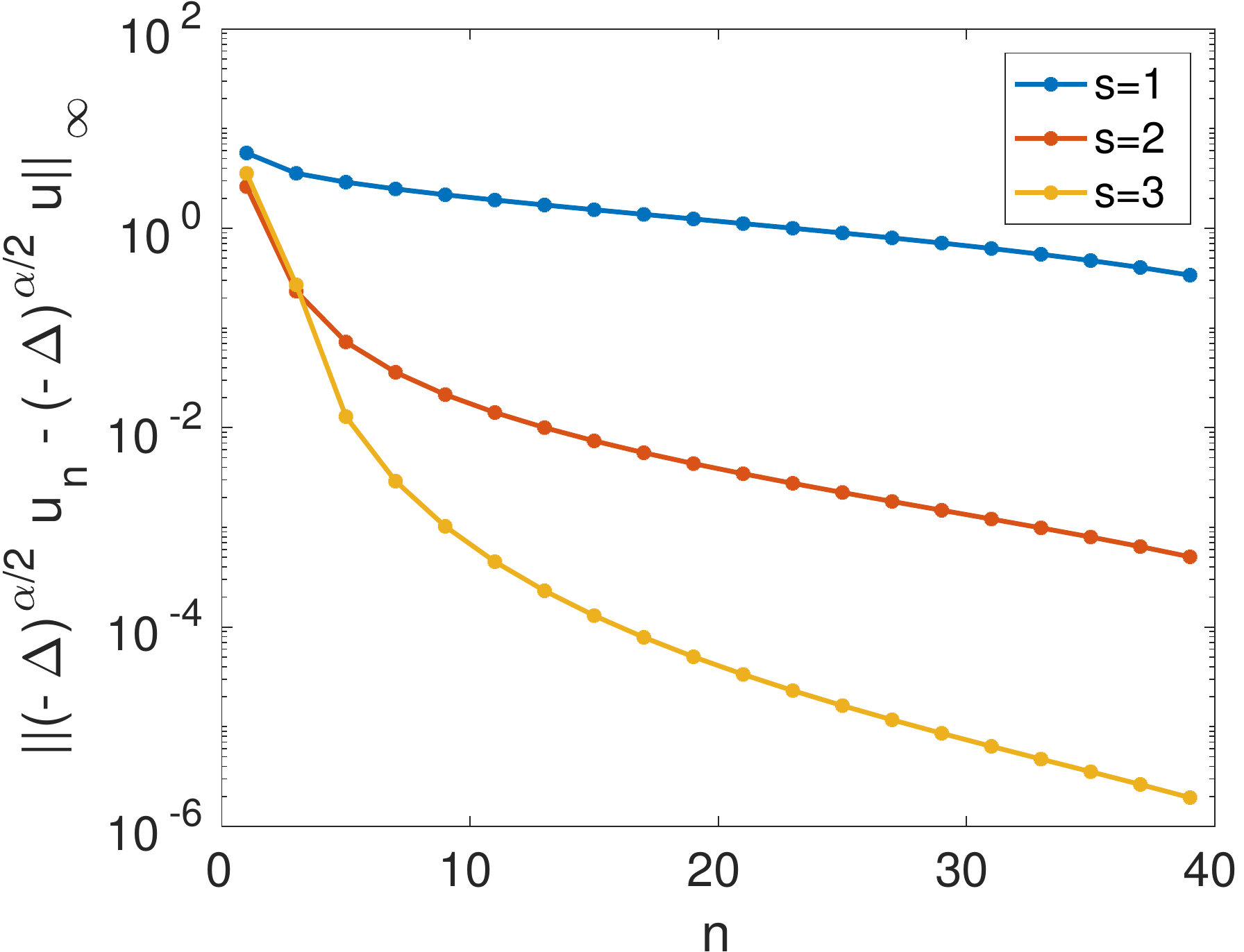}

\includegraphics[width=0.45\textwidth]{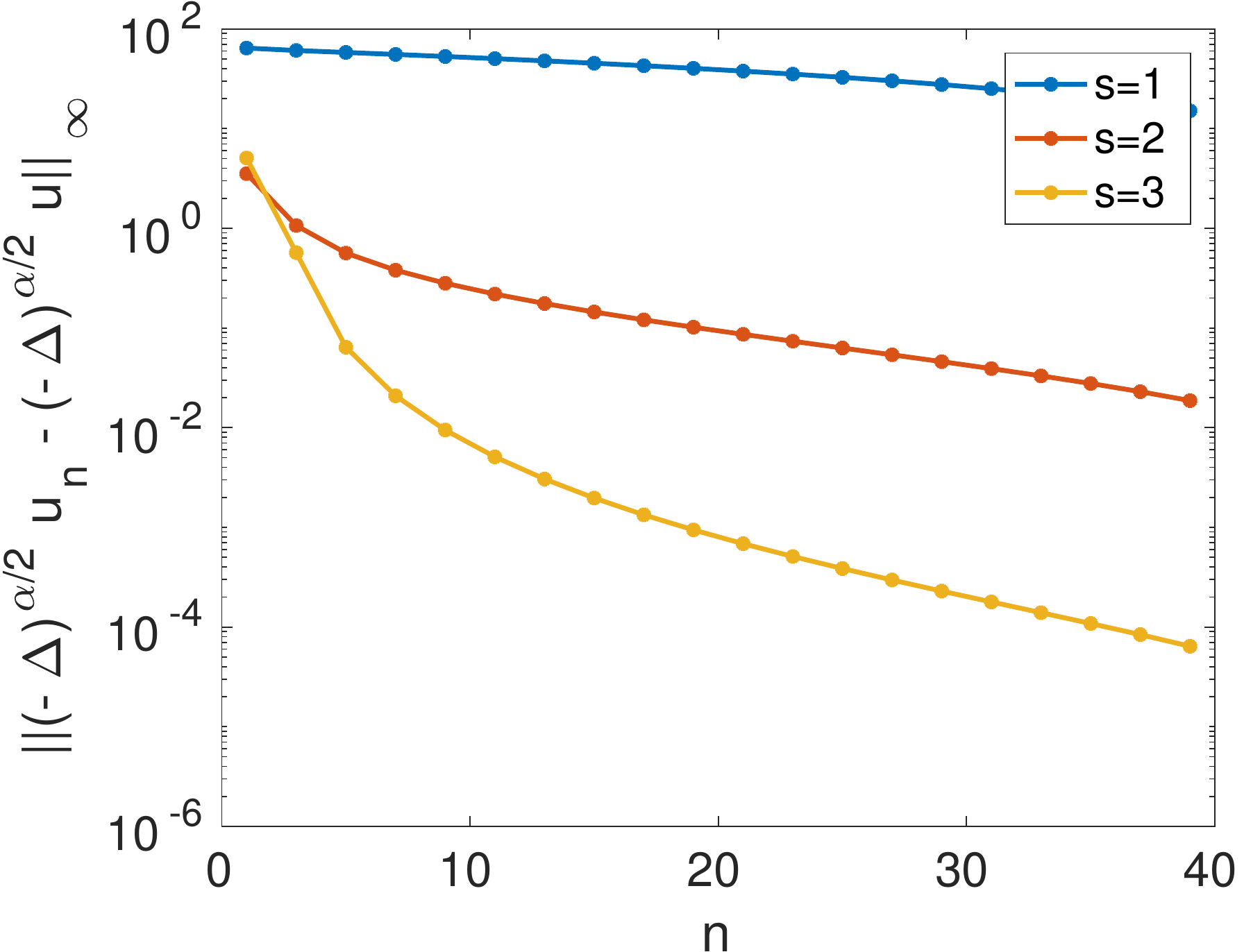}
\includegraphics[width=0.45\textwidth]{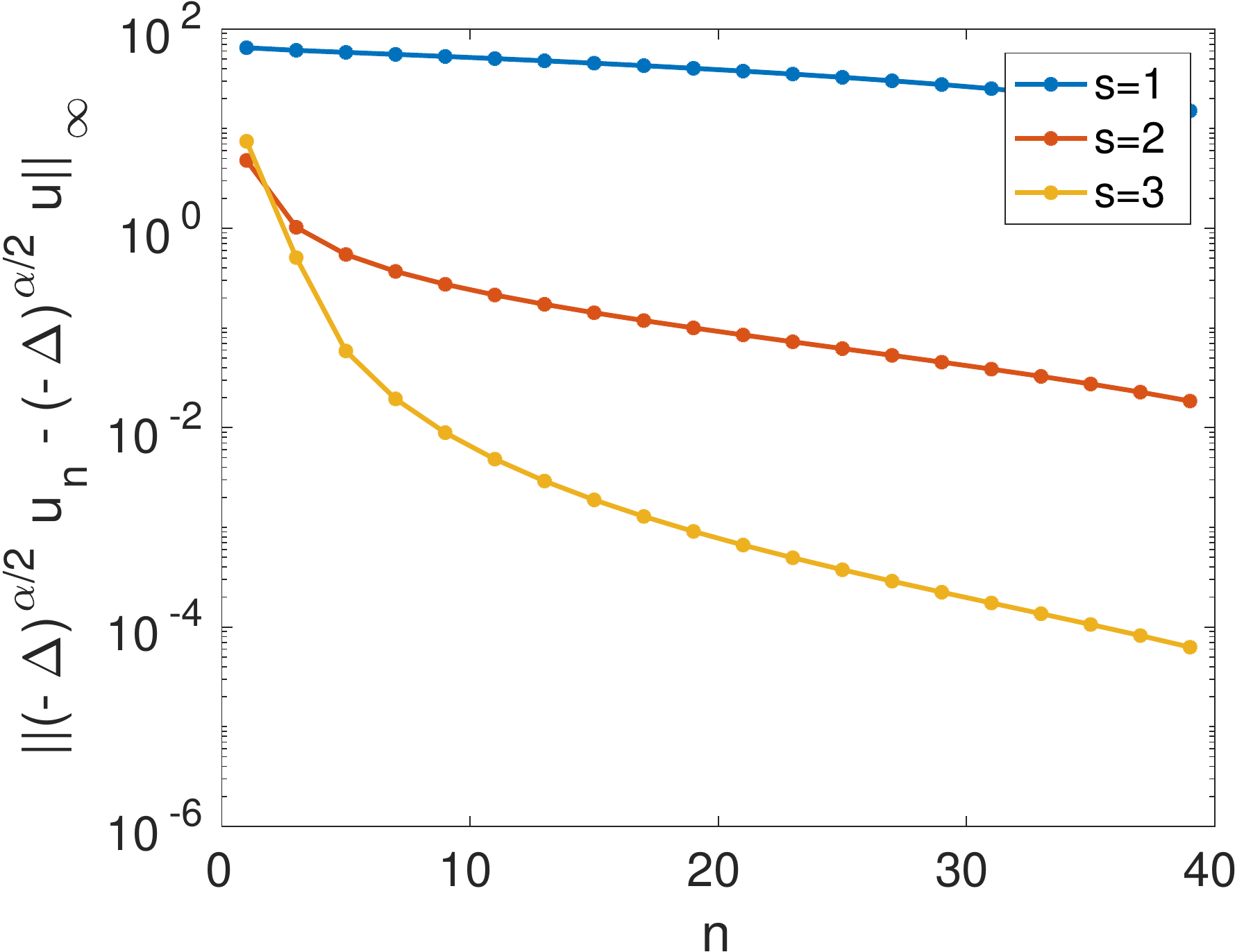}

\caption{Convergence plot for \cref{equ:testcase2}. Left column: 2D. Right column: 3D. First row: $\alpha=0.5$. Second row: $\alpha=1.0$. Third row: $\alpha=1.5$. We can see in all cases the solution convergences exponentially with respect to $n$. In the case $s=1$, the convergence is slow and the error is relatively large compared to other cases. In this case, even the first order derivative  of the function $(1-|\bx|^2)^{1-\frac{\alpha}{2}}$ is infinite on the boundary $|\bx|=1$ so polynomial approximation is not appropriate here. Therefore the slow convergence and relatively large error are expected.}
\label{fig:ex}
\end{figure}

\subsection{Solving the Fractional Poisson Equation}\label{sect:2}

In this section, we solve the homogeneous fractional Poisson equation using our spectral methods. We demonstrate the effectiveness of the method for non-radial solutions. 2D results will be provided in this section. 

The homogeneous fractional Poisson equation has been used as the benchmark problem for fractional Laplacian computation. The problem is
\begin{equation}
    \begin{aligned}
       (-\Delta)^{\frac\alpha2} u(\bx) &= f(\bx) & \bx \in B\ \\
        u(\bx) &= 0 &\bx \in B^c
    \end{aligned}
\end{equation}
where $f(\bx)\in C^k(\bar B)$ for some $k\in\mathbb{N}$. There are some known analytical solutions to the problem for specific $f(\bx)$. In our numerical experiments, we consider the cases shown in \cref{tab:foureqs}.

\begin{table}[htpb]
\centering
\begin{tabular}{l|c}
\toprule
  $u(\bx)$ in the ball&$(-\Delta)^{\frac{\alpha}{2}} u$ in the ball \\
  \midrule
 $(1-|\bx|^2)^{\frac{\alpha}{2}}$ & ${2^\alpha }\Gamma \left( {{\alpha  \over 2} + 1} \right)\Gamma \left( {{{d + \alpha } \over 2}} \right)\Gamma {\left( {{d \over 2}} \right)^{ - 1}}$ \\[4pt]
 $(1-|\bx|^2)^{\frac{\alpha}{2}+1}$ &${2^\alpha }\Gamma \left( {{\alpha  \over 2} + 2} \right)\Gamma \left( {{{d + \alpha } \over 2}} \right)\Gamma {\left( {{d \over 2}} \right)^{ - 1}}\left[ {1 - \left( {1 + {\alpha  \over d}} \right)|\bx|{^2}} \right]$  \\[4pt]
$(1-|\bx|^2)^{\frac{\alpha}{2}}\bx_d$ & ${2^\alpha }\Gamma \left( {{\alpha  \over 2} + 1} \right)\Gamma \left( {{{d + \alpha } \over 2} + 1} \right)\Gamma {\left( {{d \over 2} + 1} \right)^{ - 1}}{\bx_d}$\\[4pt]
$(1-|\bx|^2)^{\frac{\alpha}{2}+1}\bx_d$ & ${2^\alpha }\Gamma \left( {{\alpha  \over 2} + 2} \right)\Gamma \left( {{{d + \alpha } \over 2} + 1} \right)\Gamma {\left( {{d \over 2} + 1} \right)^{ - 1}}\left[ {1 - \left( {1 + {\alpha  \over {d + 2}}} \right)|\bx|{^2}} \right]\bx_d$\\
  \bottomrule
\end{tabular}
\caption{Several functions whose Laplacian can be computed analytically. The third and fourth equations are not radial.}
\label{tab:foureqs}
\end{table}

Note the third and fourth equations are not radial symmetric; this example also demonstrates the effectiveness of the method for non-radial functions~(\cref{fig:nonradial}). The error is shown in \cref{tab:four}. We see that the method is quite accurate and efficient for the problem. 


\begin{figure}[htbp]
\centering
\includegraphics[width=0.8\textwidth]{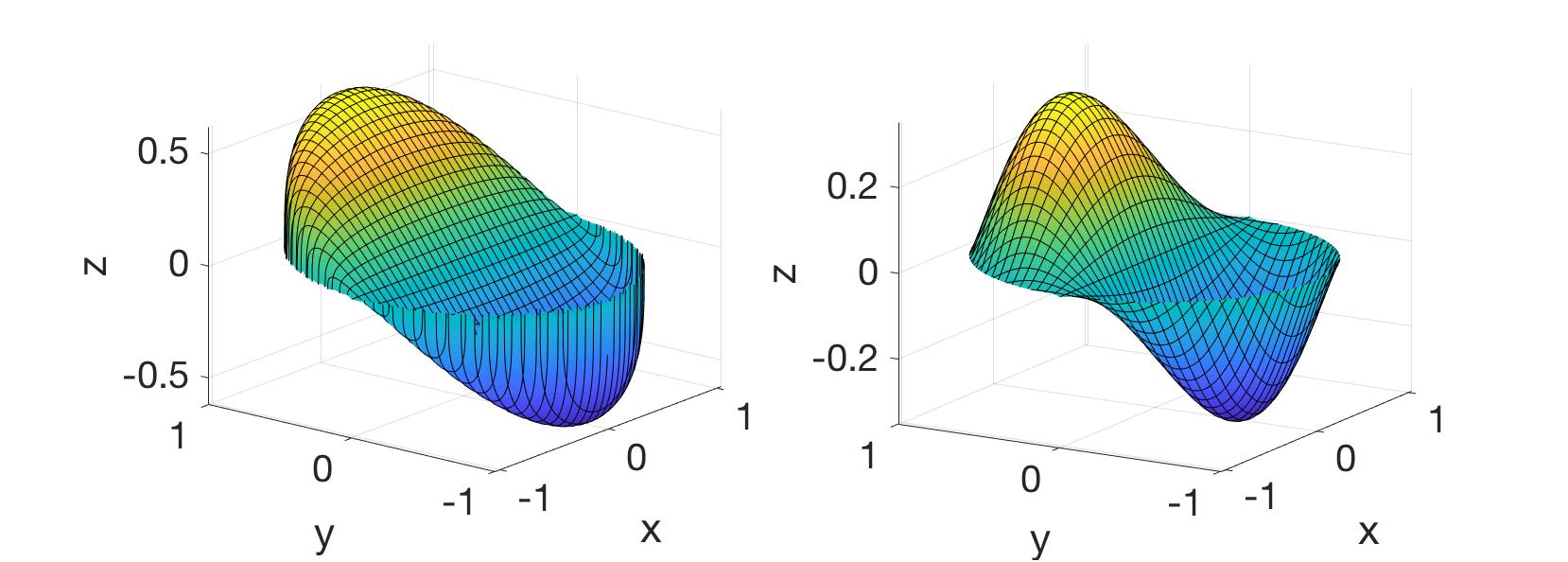}
\caption{Non-radial functions used as test cases eq.~3 and eq.~4 in \cref{tab:foureqs}.}
\label{fig:nonradial}
\end{figure}

\begin{table}[htbp]
\centering
\begin{tabular}{l|cccc}
\toprule
  Equation ($\alpha=0.5$) & $n=0$ & $n=1$ & $n=2$ \\
  \midrule
 eq.~1~$(L=0)$ & 0.0000  & 4.5652{\footnotesize $\times 10^{-13}$}  & 3.3995{\footnotesize $\times 10^{-13}$} \\
 eq.~2~$(L=0)$ & 2.5647{\footnotesize $\times 10^{-1}$}  & 9.9838{\footnotesize $\times 10^{-10}$}  & 9.9811{\footnotesize $\times 10^{-10}$} \\
eq.~3~$(L=1)$  & 2.3870{\footnotesize $\times 10^{-15}$}  & 1.6488{\footnotesize $\times 10^{-7}$}  & 2.4919{\footnotesize $\times 10^{-7}$} \\
eq.~4~$(L=1)$   & 1.7506{\footnotesize $\times 10^{-1}$}  & 1.0946{\footnotesize $\times 10^{-7}$}  & 1.3067{\footnotesize $\times 10^{-7}$} \\
\midrule
\end{tabular}

\medskip

\begin{tabular}{l|cccc}
  Equation ($\alpha=1.0$) & $n=0$ & $n=1$ & $n=2$ \\
  \midrule
 eq.~1~$(L=0)$ &   0.0000  & 0.0000  & 0.0000 \\
 eq.~2~$(L=0)$ &  1.7877{\footnotesize $\times 10^{-1}$}  & 4.3368{\footnotesize $\times 10^{-17}$}  & 3.6429{\footnotesize $\times 10^{-17}$} \\
eq.~3~$(L=1)$  &   1.0478{\footnotesize $\times 10^{-15}$}  & 1.2869{\footnotesize $\times 10^{-8}$}  & 1.5686{\footnotesize $\times 10^{-8}$} \\
eq.~4~$(L=1)$    & 1.5003{\footnotesize $\times 10^{-1}$}  & 1.2148{\footnotesize $\times 10^{-8}$}  & 1.6127{\footnotesize $\times 10^{-8}$} \\
\midrule
\end{tabular}

\medskip

\begin{tabular}{l|cccc}
  Equation ($\alpha=1.5$) & $n=0$ & $n=1$ & $n=2$ \\
  \midrule
 eq.~1~$(L=0)$ & 0.0000  & 2.3648{\footnotesize $\times 10^{-14}$}  & 1.2768{\footnotesize $\times 10^{-14}$} \\
 eq.~2~$(L=0)$  & 1.3717{\footnotesize $\times 10^{-1}$}  & 6.7892{\footnotesize $\times 10^{-12}$}  & 6.8239{\footnotesize $\times 10^{-12}$} \\
eq.~3~$(L=1)$   & 2.3870{\footnotesize $\times 10^{-15}$}  & 1.6488{\footnotesize $\times 10^{-7}$}  & 2.4919{\footnotesize $\times 10^{-7}$} \\
 eq.~4~$(L=1)$   & 1.7506{\footnotesize $\times 10^{-1}$}  & 1.0946{\footnotesize $\times 10^{-7}$}  & 1.3067{\footnotesize $\times 10^{-7}$} \\
\bottomrule
\end{tabular}
\caption{Error for test functions in \cref{tab:foureqs}. We have used different $n$ in \cref{equ:fc} and \cref{equ:ftou}. We also tried different $L$ (i.e., $c_{l0}^n$, $c_{l1}^n$, $l\geq L$ are all computed). For the first equation, only $c^0_{00}$ is needed to represent the solution and therefore it converges for all choices of $n$. For the second equation, $c^0_{00}$, $c^1_{00}$ are needed and we see convergence for $n\geq 1$. For the third and fourth equation, due to the presence of $\bx_2$, $L\geq 1$ is required to represent $\bx_2=r\sin\theta$. Note that for eq.~3, $L=1$, the error may become larger for $n\geq 1$ because of the roundoff error from the integration. For example, in the case of eq.~3, $\alpha=0.5$, $n=3$, the coefficient corresponding to $p_{1,1,2}$, i.e., $c^2_{1,1}$ should be $0$. Due to the roundoff error, we lose this orthogonality in the integration, and obtain a numerical coefficient $c^2_{1,1}\approx 2.8980\times 10^{-7}$. This error contributes to the error $2.4919\times 10^{-7}$ in the table.
}
\label{tab:four}
\end{table}

We also apply the method to some other functions. Let
\begin{equation}
    f(\bx) =|\bx|^2 \cos(16 |\bx|)
\end{equation}
which has high oscillations. The error is measured by first computing the solution $u$ using high degree $N=40$ as an approximation to the true solution. And for solutions $u_n$ using degree $n\ll N$, the error is computed by $\|u_n-u\|_\infty$. The solution and convergence plot are shown in \cref{fig:fx}.

\begin{figure}[htbp]
\centering
\includegraphics[width=1.0\textwidth]{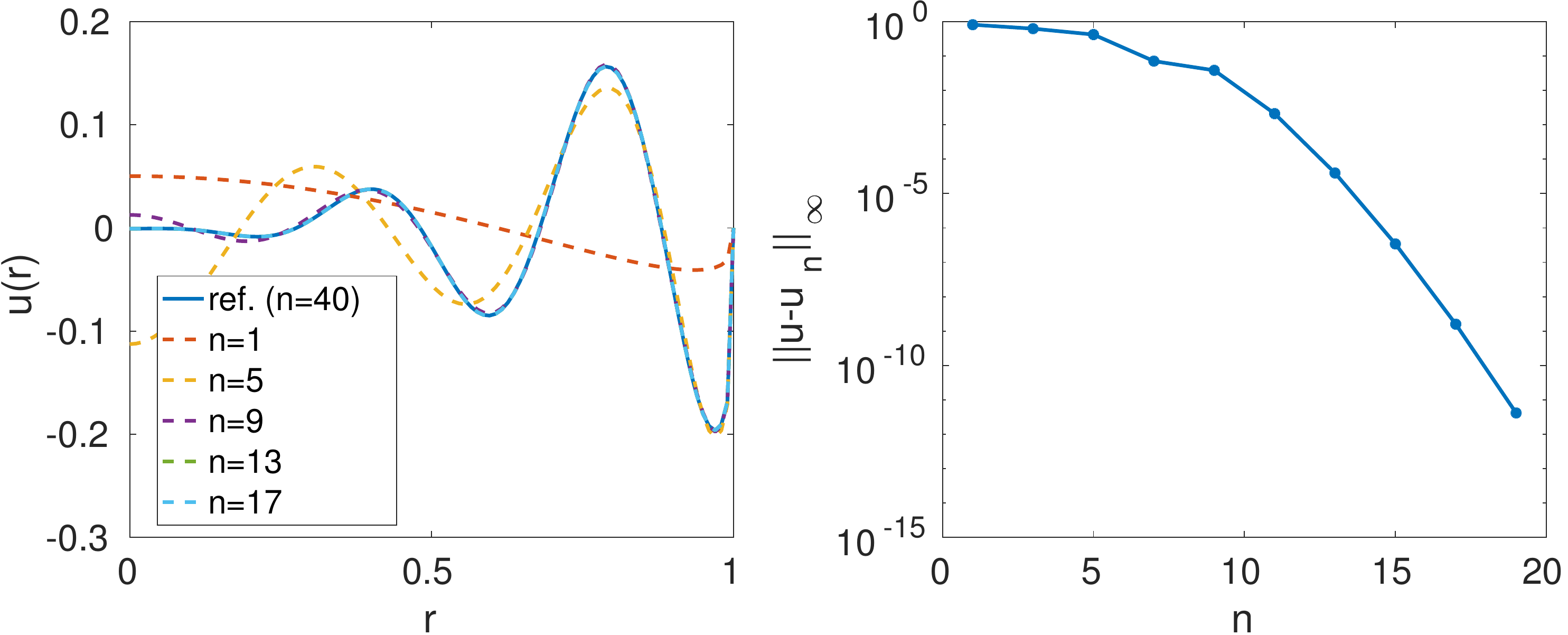}
\caption{Solution profile and convergence plot. We see that the infimum error $\|u-u_n\|_\infty$ decays exponentially as $n$ increases.}
\label{fig:fx}
\end{figure}

\subsection{Solving the Fractional Diffusion Equation}\label{sect:3}

In this subsection, we demonstrate the efficiency of the algorithm by solving the fractional diffusion equation
\begin{equation}\label{equ:diffusion}
    \begin{aligned}
        \partial_t u(\bx,t) &= -(-\Delta)^{\frac{\alpha}{2}} u, \; \bx\in B, \;  t>0\\
        u(\bx,t) &= 0, \; \bx\in B^c, \; t\geq 0\\
        u(\bx,0) &= (1-|\bx|^2)^{\alpha/2}, \; x\in B
    \end{aligned}
\end{equation}
where $B=B_0^3(1)$ is the unit ball in $\mathbb{R}^3$. Note as the initial condition is radial symmetric, the solution is also radial; for visualization, we plot the value $u(|\bx|) = u(|\bx|,1)$ for $\bx=(r,0,0), r\in(0,1)$ for different choices of $\Delta t$. 

As indicated in the error analysis, the spectral method can be very accurate, so typically we can guarantee a small error for the spatial variable. The error caused by the temporal discretization will dominate the total error. We can demonstrate that by looking at the time convergence rate, which is usually $\mathcal{O}(\Delta t)$ for the explicit Euler method.

In our example, the factor $(1-|\bx|^2)^{\alpha/2}$ in the initial condition is consistent with the regularity result for the fractional Laplacian: as $(-\Delta)^{\alpha/2}u(\bx,0)$, $\forall \bx\in B$ exists and $u(\bx,0)=0,\forall \bx\in B^c$, $\frac{u}{\delta^{\alpha/2}}\big|_\Omega$ can be continuously extended to $\bar \Omega$.

We will use the implicit time stepping method for \cref{equ:diffusion}. Let $u^k$ denotes $u(\bx, t_k)$ where $t_k=k\Delta t$, $k=$ 0, 1, 2,\ldots, we have the discretized scheme
\begin{equation}\label{equ:df}
    \frac{u^{k+1}-u^k}{\Delta t} = -(-\Delta)^{\alpha/2} u^k
\end{equation}

Plug 
\begin{equation}
    u^k = \sum_{m=0}^{N-1} c_m^{k+1} d_{m,0}^{\alpha,3}(1-r^2)^{\alpha/2} P_m^{\left( {{\alpha  \over 2},{1 \over 2}} \right)}(2{r^2} - 1)
\end{equation} 
into \cref{equ:df}, and multiply both sides by 
$$P_n^{\left( {{\alpha  \over 2},{1 \over 2}} \right)}(2{r^2} - 1) r^2$$
then integrate over $(0,1)$, we then have
\begin{equation}
    (I + \Delta t B^{-1}AD) c^{k+1} = c^k
\end{equation}
where
\begin{align}
    B &= \mathrm{diag}\Bigg(\left\{ \int_0^1 {{{(1 - {r^2})}^{\alpha /2}}\left(P_m^{\left( {{\alpha  \over 2},{1 \over 2}} \right)}(2{r^2} - 1)\right)^2{r^2}dr}  \right\}_{m=0}^{N-1}\Bigg)\\
    A &= {\left\{ {\int_0^1 {P_m^{\left( {{\alpha  \over 2},{1 \over 2}} \right)}(2{r^2} - 1)P_n^{\left( {{\alpha  \over 2},{1 \over 2}} \right)}(2{r^2} - 1){r^2}dr} } \right\}_{m,n}}\\
    D &= \mathrm{diag}\Big(\left\{ d_{m,0}^{\alpha,3}  \right\}_{m=0}^{N-1}\Big)\\
    c^k &= \{c^k_n\}_{n=0}^{N-1}
\end{align}

$c^0$ is computed by expand $u(\bx,0)$ with a truncated series \cref{equ:uexpansion3}. In our experiment, we use $N=10$.

The stability of the scheme is easily established by observing that $B$, $D$ are diagonal matrices with positive entries. In addition, $A$ is symmetric positive definite by observing that it is the Gram matrix of the vector
$$\left\{ {P_n^{\left( {{\alpha  \over 2},{1 \over 2}} \right)}(2{r^2} - 1)} \right\}_{n = 0}^{N - 1}$$
with respect to $L^2(w)$ where the weight function is $w=r^2$. A standard argument leads to the $\mathcal{O}(\Delta t)$ convergence, which is also demonstrated numerically in \cref{fig:diffusion}.

\begin{figure}[htbp]
\centering
\includegraphics[width=0.45\textwidth]{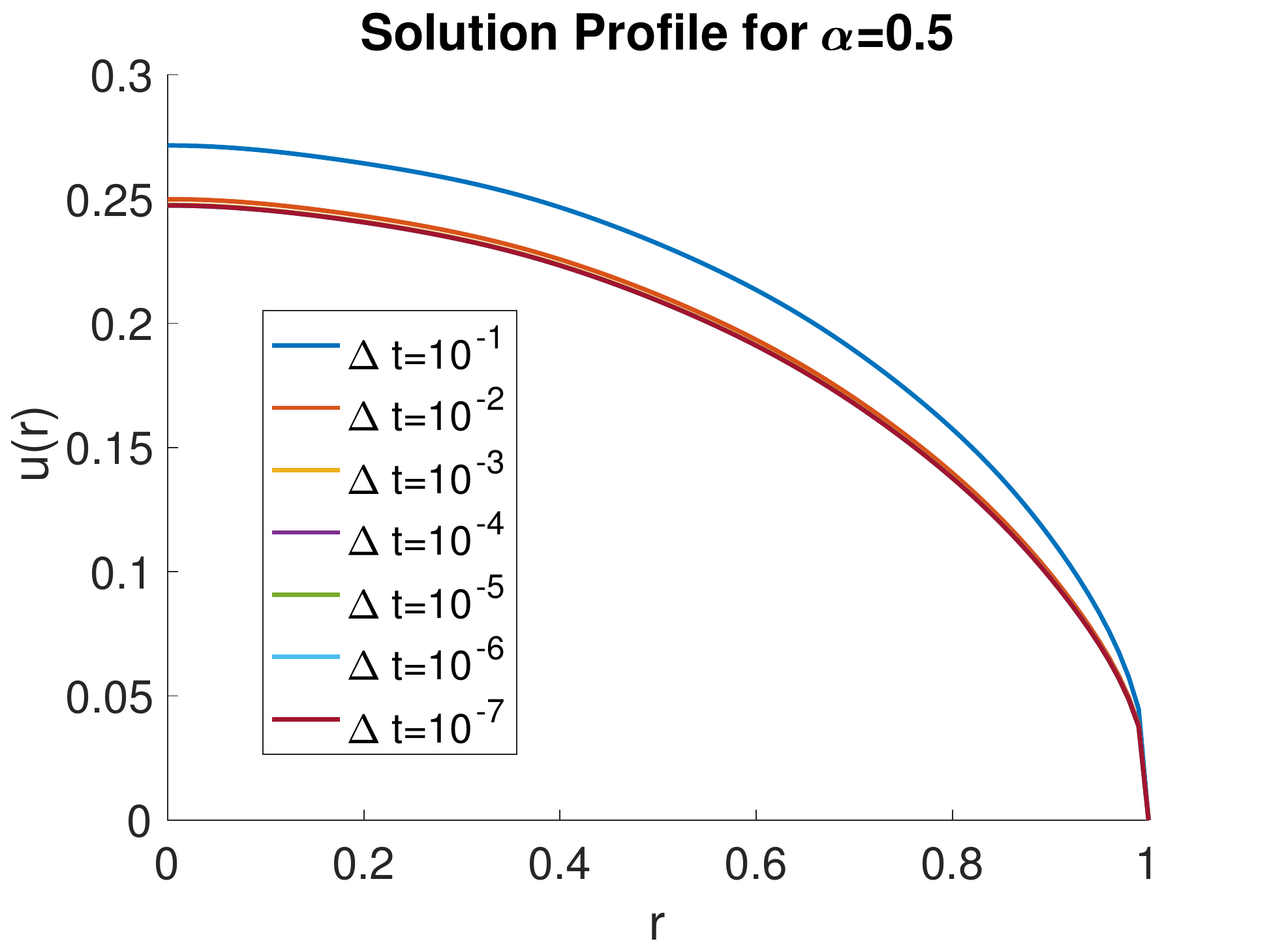}
\includegraphics[width=0.45\textwidth]{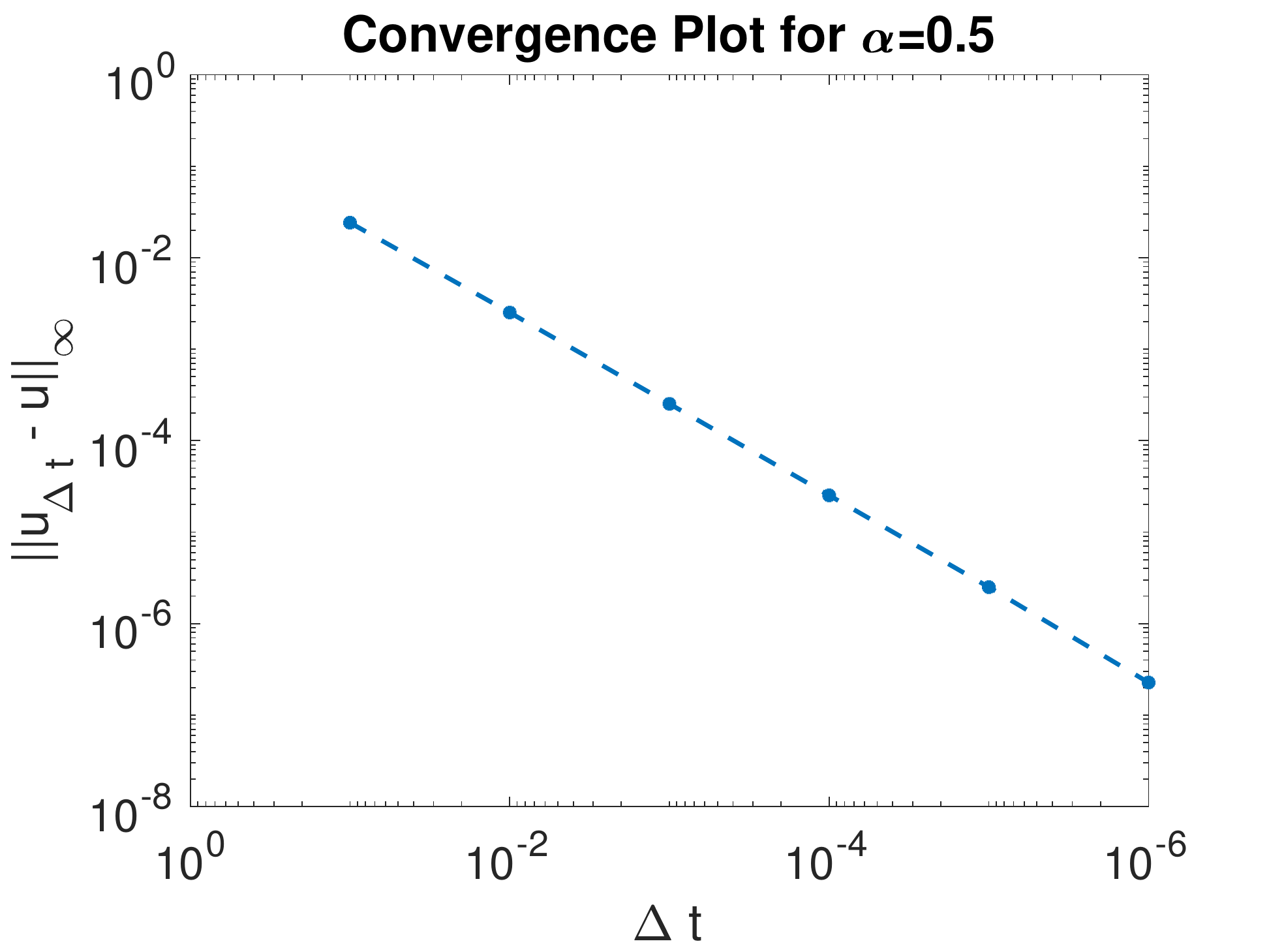}
\includegraphics[width=0.45\textwidth]{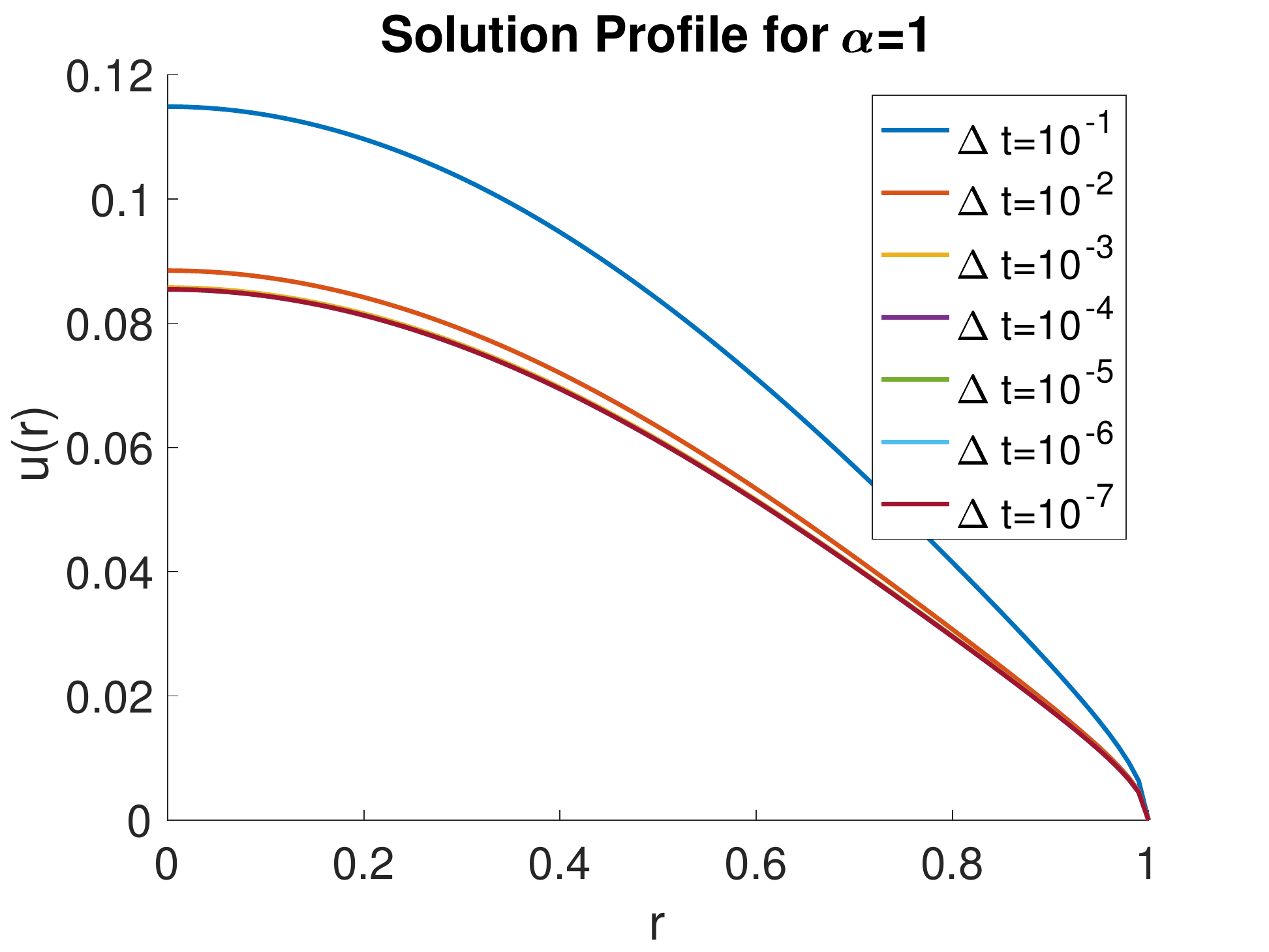}
\includegraphics[width=0.45\textwidth]{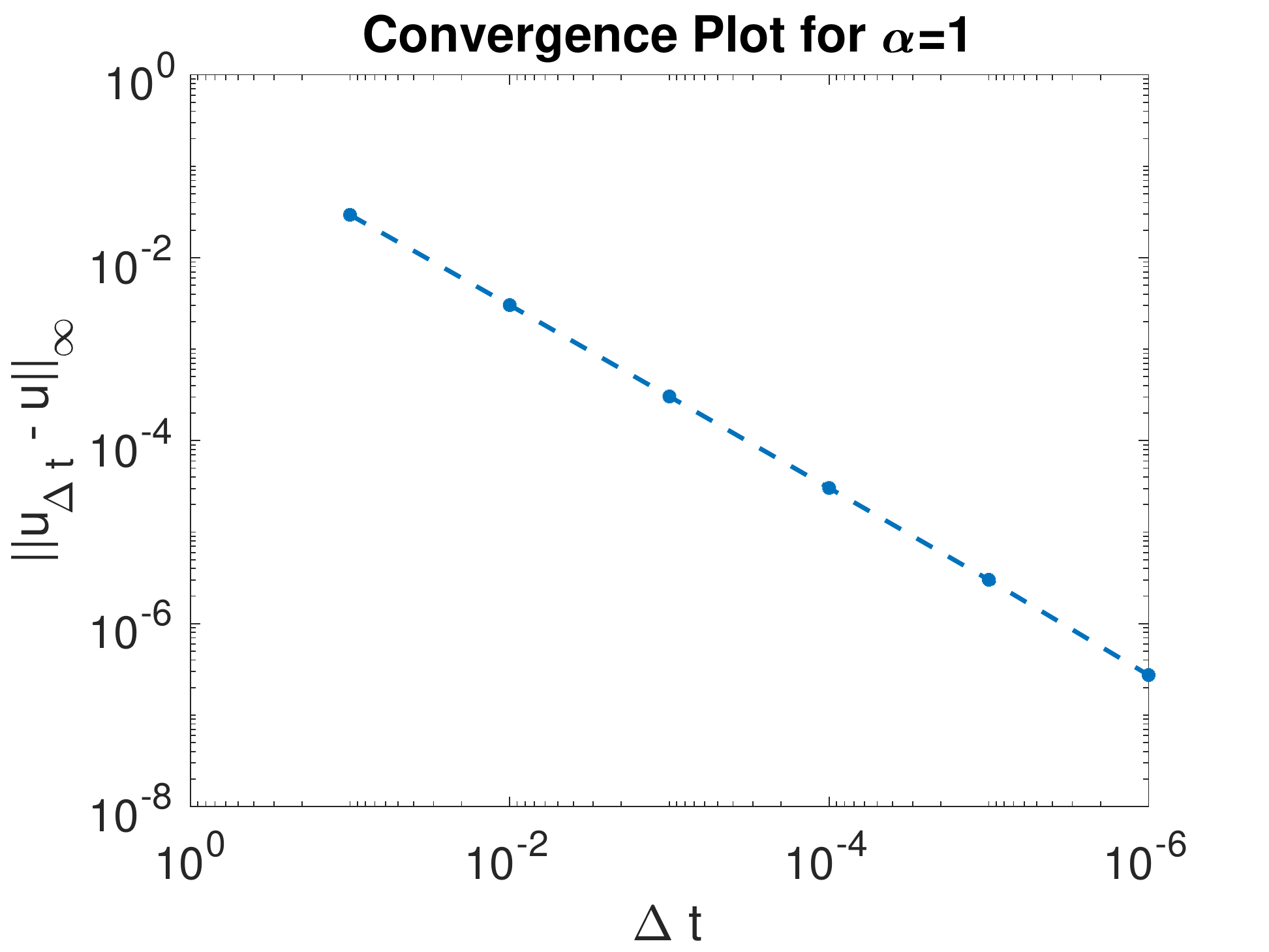}
\includegraphics[width=0.45\textwidth]{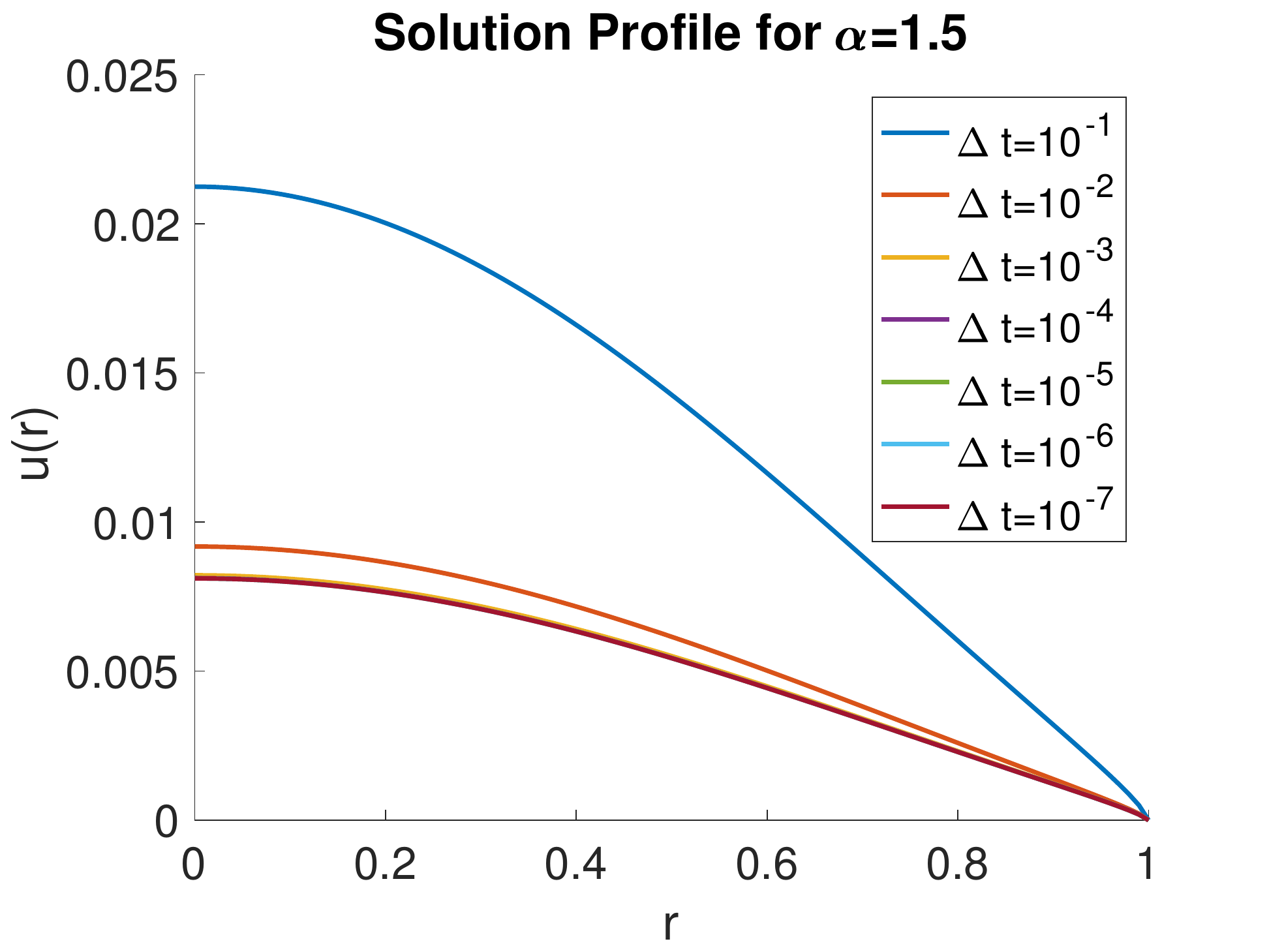}
\includegraphics[width=0.45\textwidth]{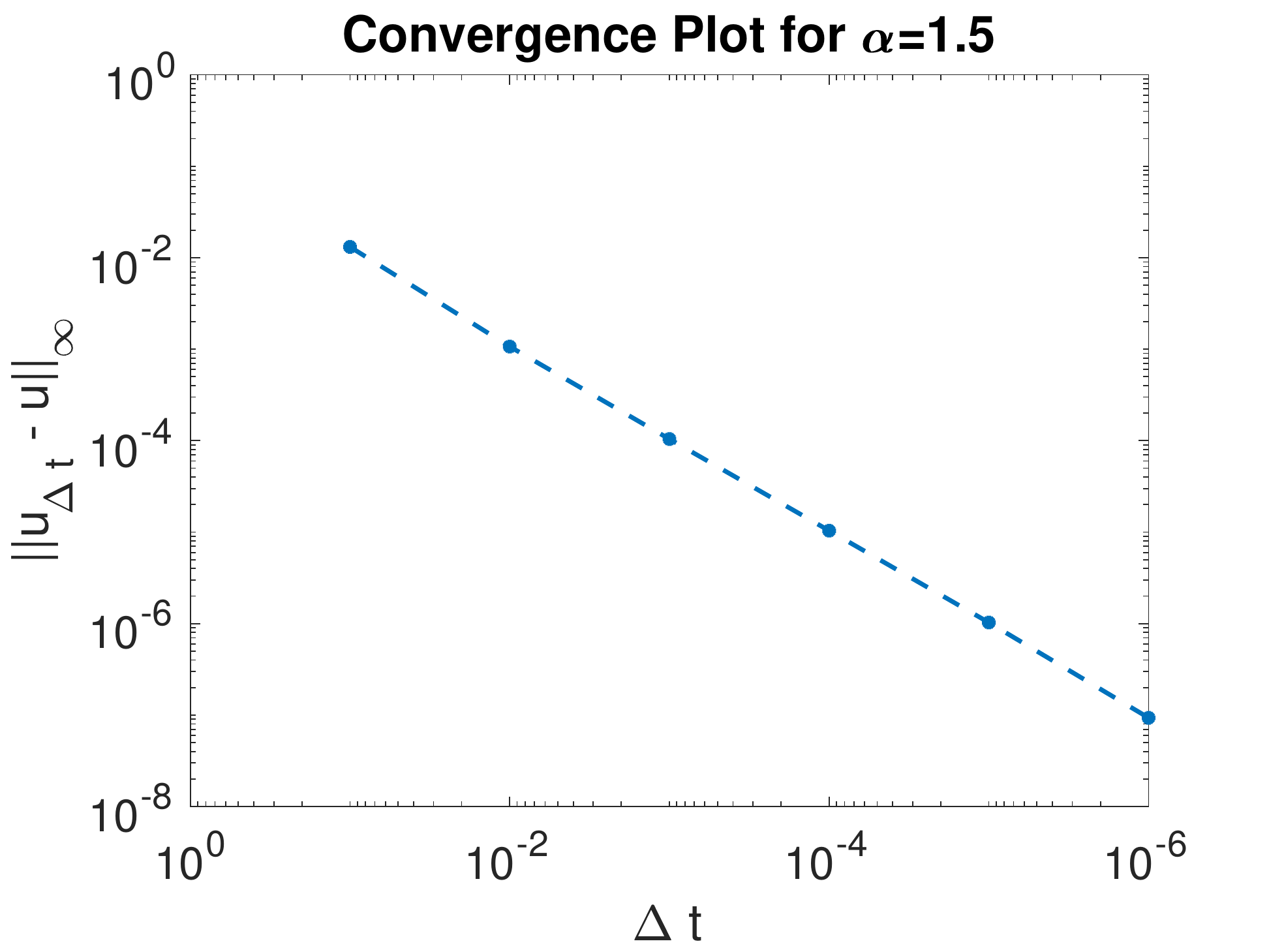}
\caption{Solution to \cref{equ:diffusion} for $\alpha=0.5,1.0,1.5$. The left column shows the solution $u(|\bx|) = u(|\bx|,1)$ at time $T=1$; the right column shows the solution error $\|u_{\Delta t} - u\|_\infty$. The convergence rate is about $\mathcal{O}(\Delta t)$ in all cases, which is the same as the implicit method for the normal diffusion equation. Note that the convergence rate is $\alpha$-independent despite the fact that the initial condition is only H\"older continuous.}
\label{fig:diffusion}
\end{figure}

\section{Conclusion}

In this paper, we proposed and implemented spectral methods for approximating the fractional Laplacian and solving the fractional Poisson problem in 2D and 3D unit balls. It is based on a recent result in \cite{dyda2017fractional}. The spectral methods can be accurate if the solution or the right hand side has the correct regularity. In general, we can only expect H\"older continuity for the solution to a fractional Laplace problem. It causes trouble when we apply finite difference or finite element methods, which assume the existence of higher order derivatives or that the solution can be well represented by polynomials. The spectral method, however, avoids the reduced regularity problem by the explicit formulation of the eigenfunctions and eigenvalues of the fractional Laplacian in the unit balls under the weighted $L^2$ space. The resulting method enjoys spectral accuracy for all fractional index $\alpha\in (0,2)$ and is computationally efficient due to the orthogonality of the basis functions. Numerical experiments show consistency with our analysis. 

\bibliographystyle{mybib}
\bibliography{biblio}

\begin{thebibliography}{10}

\bibitem{chen2004soft}
W.~Chen.
\newblock \titlecap{Soft Matter And Fractional Mathematics: Insights Into
  Mesoscopic Quantum And Time-Space Structures}.
\newblock {\em arXiv preprint cond-mat/0405345}, 2004.

\bibitem{dipierro2015dislocation}
S.~Dipierro, G.~Palatucci, and E.~Valdinoci.
\newblock \titlecap{Dislocation Dynamics In Crystals: A Macroscopic Theory In A
  Fractional Laplace Setting}.
\newblock {\em Communications in Mathematical Physics}, 333(2):1061--1105,
  2015.

\bibitem{bakunin2008turbulence}
O.~G. Bakunin.
\newblock {\em Turbulence and Diffusion: Scaling versus Equations}.
\newblock Springer Science \& Business Media, 2008.

\bibitem{bologna2000anomalous}
M.~Bologna, C.~Tsallis, and P.~Grigolini.
\newblock \titlecap{Anomalous diffusion associated with nonlinear fractional
  derivative Fokker-Planck-like equation: Exact time-dependent solutions}.
\newblock {\em Physical Review E}, 62(2):2213, 2000.

\bibitem{woyczynski2001levy}
W.~A. Woyczy\'nski.
\newblock \titlecap{L\'evy processes in the physical sciences}.
\newblock \titlecap{L\'evy processes in the physical sciences}, pages 241--266.
  Springer, 2001.

\bibitem{gatto2015numerical}
P.~Gatto and J.~S. Hesthaven.
\newblock \titlecap{Numerical approximation of the fractional {L}aplacian via
  hp-finite elements, with an application to image denoising}.
\newblock {\em Journal of Scientific Computing}, 65(1):249--270, 2015.

\bibitem{vazquez2012nonlinear}
J.~L. V\'azquez.
\newblock \titlecap{Nonlinear diffusion with fractional {L}aplacian operators}.
\newblock \titlecap{Nonlinear diffusion with fractional {L}aplacian operators},
  pages 271--298. Springer, 2012.

\bibitem{Chen2004}
W.~Chen and S.~Holm.
\newblock \titlecap{Fractional {L}aplacian time-space models for linear and
  nonlinear lossy media exhibiting arbitrary frequency power-law dependency}.
\newblock {\em The Journal of the Acoustical Society of America},
  115(4):1424--1430, 2004.

\bibitem{cont2005finite}
R.~Cont and E.~Voltchkova.
\newblock \titlecap{A finite difference scheme for option pricing in jump
  diffusion and exponential L{\'e}vy models}.
\newblock {\em SIAM Journal on Numerical Analysis}, 43(4):1596--1626, 2005.

\bibitem{epps2018turbulence}
B.~P. Epps and B.~Cushman-Roisin.
\newblock \titlecap{Turbulence Modeling via the Fractional {L}aplacian}.
\newblock {\em arXiv preprint arXiv:1803.05286}, 2018.

\bibitem{ros2014integro}
X.~Ros et~al.
\newblock \titlecap{Integro-differential equations: Regularity theory and
  Pohozaev identities}.
\newblock 2014.

\bibitem{lin2016isogeometric}
H.~Lin, Y.~Xiong, and Q.~Hu.
\newblock \titlecap{Isogeometric least-squares collocation method with
  consistency and convergence analysis}.
\newblock {\em arXiv preprint arXiv:1601.07244}, 2016.

\bibitem{duo2017comparative}
S.~Duo, H.~Wang, and Y.~Zhang.
\newblock \titlecap{A comparative study on nonlocal diffusion operators related
  to the fractional Laplacian}.
\newblock {\em arXiv preprint arXiv:1711.06916}, 2017.

\bibitem{2018arXiv180203770M}
V.~{Minden} and L.~{Ying}.
\newblock \titlecap{A simple solver for the fractional Laplacian in multiple
  dimensions}.
\newblock {\em ArXiv e-prints}, February 2018.

\bibitem{duo2018novel}
S.~Duo, H.~W. van Wyk, and Y.~Zhang.
\newblock \titlecap{A novel and accurate finite difference method for the
  fractional Laplacian and the fractional Poisson problem}.
\newblock {\em Journal of Computational Physics}, 355:233--252, 2018.

\bibitem{ainsworth2017towards}
M.~Ainsworth and C.~Glusa.
\newblock \titlecap{Towards an efficient finite element method for the integral
  fractional Laplacian on polygonal domains}.
\newblock {\em arXiv preprint arXiv:1708.01923}, 2017.

\bibitem{acosta2017fractional}
G.~Acosta and J.~P. Borthagaray.
\newblock \titlecap{A fractional Laplace equation: Regularity of solutions and
  finite element approximations}.
\newblock {\em SIAM Journal on Numerical Analysis}, 55(2):472--495, 2017.

\bibitem{lu2017spectral}
H.~Lu, P.~W. Bates, W.~Chen, and M.~Zhang.
\newblock \titlecap{The spectral collocation method for efficiently solving
  PDEs with fractional Laplacian}.
\newblock {\em Advances in Computational Mathematics}, pages 1--18, 2017.

\bibitem{acosta2018regularity}
G.~Acosta, J.~P. Borthagaray, O.~Bruno, and M.~Maas.
\newblock \titlecap{Regularity theory and high order numerical methods for the
  (1d)-fractional Laplacian}.
\newblock {\em {M}athematics of {C}omputation}, 87(312):1821--1857, 2018.

\bibitem{khader2011numerical}
M.~Khader.
\newblock \titlecap{On the numerical solutions for the fractional diffusion
  equation}.
\newblock {\em Communications in Nonlinear Science and Numerical Simulation},
  16(6):2535--2542, 2011.

\bibitem{hanert2010comparison}
E.~Hanert.
\newblock \titlecap{A comparison of three Eulerian numerical methods for
  fractional-order transport models}.
\newblock {\em Environmental fluid mechanics}, 10(1-2):7--20, 2010.

\bibitem{zayernouri2014fractional}
M.~Zayernouri and G.~E. Karniadakis.
\newblock \titlecap{Fractional spectral collocation method}.
\newblock {\em SIAM Journal on Scientific Computing}, 36(1):A40--A62, 2014.

\bibitem{bueno2014fourier}
A.~Bueno-Orovio, D.~Kay, and K.~Burrage.
\newblock \titlecap{Fourier spectral methods for fractional-in-space
  reaction-diffusion equations}.
\newblock {\em BIT Numerical Mathematics}, 54(4):937--954, 2014.

\bibitem{ros2014dirichlet}
X.~Ros-Oton and J.~Serra.
\newblock \titlecap{The Dirichlet problem for the fractional Laplacian:
  regularity up to the boundary}.
\newblock {\em Journal de Math{\'e}matiques Pures et Appliqu{\'e}es},
  101(3):275--302, 2014.

\bibitem{johansson2016computing}
F.~Johansson.
\newblock \titlecap{Computing hypergeometric functions rigorously}.
\newblock {\em arXiv preprint arXiv:1606.06977}, 2016.

\bibitem{dyda2017eigenvalues}
B.~Dyda, A.~Kuznetsov, and M.~Kwa{\'s}nicki.
\newblock \titlecap{Eigenvalues of the fractional Laplace operator in the unit
  ball}.
\newblock {\em Journal of the London Mathematical Society}, 95(2):500--518,
  2017.

\bibitem{saborid2008coordinate}
M.~P. Saborid.
\newblock \titlecap{The Coordinate-free Approach to Spherical Harmonics}.
\newblock {\em arXiv preprint arXiv:0806.3367}, 2008.

\bibitem{ferrers1877elementary}
N.~M. Ferrers.
\newblock {\em An Elementary Treatise on Spherical Harmonics and Subjects
  Connected with Them}.
\newblock Macmillan and Company, 1877.

\bibitem{dyda2017fractional}
B.~Dyda, A.~Kuznetsov, and M.~Kwa{\'s}nicki.
\newblock \titlecap{Fractional Laplace operator and Meijer G-function}.
\newblock {\em Constructive Approximation}, 45(3):427--448, 2017.

\bibitem{dunkl2014orthogonal}
C.~F. Dunkl and Y.~Xu.
\newblock {\em Orthogonal Polynomials of Several Variables}.
\newblock Number 155. Cambridge University Press, 2014.

\bibitem{spherical:online}
\titlecap{Spherical Harmonics and Homogeneous Harmonic Polynomials}.
\newblock \url{www.math.utk.edu/~freire/m435f07/m435sphericalharmonics.pdf}.
\newblock (Accessed on 07/09/2018).

\bibitem{ragozin1971constructive}
D.~L. Ragozin.
\newblock \titlecap{Constructive polynomial approximation on spheres and
  projective spaces.}
\newblock {\em Transactions of the American Mathematical Society},
  162:157--170, 1971.

\bibitem{glau2016improved}
K.~Glau and M.~Mahlstedt.
\newblock \titlecap{Improved error bound for multivariate Chebyshev polynomial
  interpolation}.
\newblock {\em arXiv preprint arXiv:1611.08706}, 2016.

\bibitem{warn720026:online}
\titlecap{Nonclassical Gaussian Quadrature Rules}.
\newblock \url{https://m.eet.com/media/1158807/911_part4.pdf}.
\newblock (Accessed on 06/30/2018).

\bibitem{golub1969calculation}
G.~H. Golub and J.~H. Welsch.
\newblock \titlecap{Calculation of Gauss quadrature rules}.
\newblock {\em Mathematics of computation}, 23(106):221--230, 1969.

\end{thebibliography}

\end{document}